\documentclass[A4]{article}
\usepackage{amsmath,amsthm,amssymb,enumitem,float,verbatim,tikz,mathcomp,comment,hyperref}
\usetikzlibrary{matrix,patterns,arrows}
\usepackage[outline]{contour}
\contourlength{1.2pt}

\definecolor{red}{RGB}{200,0,0}
\definecolor{green}{RGB}{0,120,0}
\definecolor{blue}{RGB}{0,0,180}
\definecolor{yellow}{RGB}{255,191,0}

\newtheorem{theorem}{Theorem}
\newtheorem{corollary}[theorem]{Corollary}
\newtheorem{lemma}[theorem]{Lemma}
\newtheorem{definition}[theorem]{Definition}
\newtheorem{remark}[theorem]{Remark}

\newcommand{\myfrac}[2]{\mbox{\footnotesize $\dfrac{#1}{#2}$}}
\newcommand{\as}[2]{\mbox{\footnotesize $\dfrac{#1+\mbox{\scriptsize $\displaystyle \tfrac 1 2$}}{\mbox{\footnotesize $#2$}}$}}
\newcommand{\al}[1]{\mbox{\footnotesize $\dfrac{\Z+\mbox{\scriptsize $\displaystyle \tfrac 1 2$}}{\mbox{\footnotesize $#1$}}$}}

\newcommand{\almin}[1]{\mbox{\footnotesize $\dfrac{\Z-\mbox{\scriptsize $\displaystyle \tfrac 1 2$}}{\mbox{\footnotesize $#1$}}$}}

\newcommand{\Rho}{\mathrm{P}}

\def\R{{\mathbb R}}
\def\Z{{\mathbb Z}}
\def\N{{\mathbb N}}
\def\C{{\mathbb C}}

\def\x{{\mathbf x}}

\def\E{{\mathcal E}}

\def\CHI{\hbox{\raise .5ex \hbox{$\chi$}}}
\def\pause{ {} }

\tikzset{
    invisible/.style={opacity=0,text opacity=0},
    visible on/.style={alt={#1{}{invisible}}},
    alt/.code args={<#1>#2#3}{%
      \alt<#1>{\pgfkeysalso{#2}}{\pgfkeysalso{#3}} % \pgfkeysalso doesn't change the path
    },
  }

\begin{document}

\title{Exponential bases for partitions of intervals}
\author{G\"otz Pfander \\ Katholische Universit\"at Eichst\"att 
   \and Shauna Revay \\ George Mason University \and David Walnut\\ George Mason University }

\maketitle
\begin{center}

\begin{minipage}[c]{16cm}
Abstract. \ \ For a partition of $[0,1]$ into intervals
$I_1,\ldots,I_n$   
we prove the existence of a partition of $\Z$ into
$\Lambda_1,\ldots, \Lambda_n$ %$\{\Lambda_k\}_{k=1}^n$
such that the complex exponential functions with frequencies in $ \Lambda_k$ form a Riesz basis for $L^2(I_k)$,
and furthermore, that for any $J\subseteq\{1,\,2,\,\dots,\,n\}$, the exponential functions with frequencies in
$ \bigcup_{j\in J}\Lambda_j$ form a Riesz basis for $L^2(I)$
for any interval $I$ with length $|I|=\sum_{j\in J}|I_j|$.
The construction extends to infinite partitions of $[0,1]$, but with size
limitations on the subsets $J\subseteq \Z$; it combines the ergodic properties of subsequences of $\Z$
known as Beatty-Fraenkel sequences with a theorem of Avdonin
on exponential Riesz bases.
\end{minipage}
\end{center}

\section{Introduction and main results}

The foundational stone of Fourier analysis  is  the statement that $\mathcal E(\Z)=\{e^{2\pi i n (\cdot)}\}_{n\in\Z}$ forms an orthonormal bases for the space of square integrable functions on the unit interval, $L^2[0,1]$.  Simple scaling and translation arguments imply the curious fact that $\mathcal E(\Z)$ splits into $\mathcal E(2\Z)$ and $\mathcal E(2\Z+1)$, which form orthogonal bases for $L^2[0,1/2]$ and $L^2[1/2,1]$, respectively.
Using the necessary relaxation of orthogonal bases  to Riesz bases, see Definition~\ref{def:RB} below, we show that any separation of the unit interval into finitely many subintervals allows for a corresponding partition of the integers.

\begin{theorem}\label{thm:main1}
For $a_0=0<a_1<a_2<\ldots<a_n=1$ exist  pairwise disjoint sets $\Lambda_1,\Lambda_2,\ldots,\Lambda_n\subseteq \Z$ 
with $\Lambda_1\cup \Lambda_2\cup \ldots\cup \Lambda_n=\Z$ so that $\E( \Lambda_{k})$ is a Riesz basis for $L^2[a_{k-1},a_{k}]$.
\end{theorem}

Theorem~\ref{thm:main1}  is a special case of the main result of this paper:

\begin{theorem}\label{thm:main2}
  For $b_1,\ldots,b_n>0$ with $\sum_{j=1}^n b_j=1$ exist pairwise disjoint sets $\Lambda_1,\ldots,\Lambda_n\subseteq \Z$ with   $\bigcup_{j=1}^n \Lambda_j=\Z$ and the property that $\E\big(\bigcup_{j\in J}\Lambda_i\big)$ is a Riesz basis for $L^2(I)$ for $I$ any interval of length $\sum_{j \in J} b_j$  for any $J\subseteq \{1,\ldots,n\}$.
\end{theorem}

Note that in general, it is not true that Riesz bases of exponentials
can be combined as above, that is, if $\E(\Lambda_1)$ and $\E(\Lambda_2)$ form Riesz bases for $L^2(I_1)$ and $L^2(I_2)$
respectively, it need not follow that $\E(\Lambda_1\,\cup\,\Lambda_2)$ forms a Riesz basis for $L^2(I_1\,\cup\,I_2)$, even if $\Lambda_1$ and $\Lambda_2$ are disjoint\footnote{When we say that a system of exponentials
$\E(\Lambda) = \{e^{2\pi i\lambda t}\colon \lambda\in\Lambda\}$ forms a Riesz basis for $L^2(I)$ we mean that the system
$\{e^{2\pi i\lambda t}\CHI_{I}(t)\colon\lambda\in\Lambda\}$ ($\CHI_I$ is the indicator function of the set $I$) forms such a basis. Consequently, the statement that $\E(\Lambda_1 \cup \Lambda_2)$ is a Riesz basis for $L^2(I_1\cup I_2)$ cannot be rephrased in terms of unions of Riesz bases for sums of subspaces of the Hilbert space $L^2(\R)$.} 
(see Remark~\ref{rem:counteregkadec}).
 
Theorem~\ref{thm:main2} generalizes to a countable set of intervals as follows.

\begin{theorem}\label{thm:maincountable}
Let $b_1,b_2,\ldots>0$ satisfy $\sum_{j=1}^\infty b_j=1$. For fixed $K\in \N$ there exist pairwise disjoint sets 
$\Lambda_1,\Lambda_2,\ldots \subseteq \Z$ with the property that for any $J\subseteq \N$ with $|J|\leq K$
we have $\E\big(\bigcup_{j\in J}\Lambda_j\big)$ is a Riesz basis for $L^2(I)$ for $I$ any interval of length $\sum_{j\in J} b_j$.
\end{theorem}

In contrast with Theorem~\ref{thm:main2}, Theorem~\ref{thm:maincountable} does not guarantee  subsets $\Lambda_j$, $j\in\N$, whose union is $\Z$.  This phenomenon in the countble setting is visible in the fact that the disjoint sets $2\Z+1, 4\Z+2, 8\Z+4, \ldots$ have the property that the sets $\mathcal E(2^j\Z+2^{j-1})$ form  orthogonal bases   for $L^2[2^{-j},2^{-j+1}]$, $j\in \N$, 
 but
 \begin{align*}
 \mathcal E\Big(\bigcup_{j=1}^\infty 2^j\Z+2^{j-1}\Big)=E\Big(\Z\setminus \{0\}\Big)\text{ is not a Riesz bases for }
   L^2\Big(\bigcup_{k=1}^\infty [2^{-j},2^{-j+1}]\Big)=L^2[0,1].
 \end{align*}

Note that Theorem~\ref{thm:main1} for a countable family holds verbatim by choosing $K=1$ in Theorem~\ref{thm:maincountable}. 
An immediate consequence of Theorem~\ref{thm:main2} is

\begin{corollary}
 Let $b_1,\ldots,b_n>0$ satisfy $\sum_{i=j}^n b_j\leq B$. Then exist pairwise disjoint sets $\Lambda_1,\ldots,\Lambda_n\subseteq \frac 1 B\Z$ with   the property that for any $J\subseteq \{1,\ldots,n\}$ we have $\E\big(\bigcup_{j\in J}\Lambda_j\big)$ is a Riesz basis for $L^2(I)$ for $I$ any interval of length $\sum_{j\in J} b_j$.
\end{corollary}

\vspace{.5cm}
Aside from describing related results from the literature in Section~\ref{sec:related} and stating some of our results in terms of sampling theory in Section~\ref{sec:sampling}, 
the body of this paper is dedicated to giving detailed proofs of our results.

The proof of Theorem~\ref{thm:main2} relies on utilizing three theorems from three branches of mathematics: the analytic Avdonin Theorem, 
a number theoretic result on Beatty-Fraenkel sequences and the probabilistic Weyl-Khinchin equidistribution theorem.  
In Section~\ref{sec:3results} we state customized version of these and illustrate their interplay to obtain in Section~\ref{sec:combiningABW} 
a proof of Theorem~\ref{thm:main1} for $n=2$ 
in the case of subintervals of irrational length. 
In Section~\ref{sec:twointervals} we present a proof of Theorem~\ref{thm:main1} for $n=2$ utilizing a different
approach than the proof in Section~\ref{sec:combiningABW} that will be used in the proofs of Theorems~\ref{thm:main2} and \ref{thm:maincountable} 
in subsequent sections.  The proofs of the main results in this paper rely on the concept of Avdonin maps which we introduce in Section~\ref{sec:AvdoninMaps}.  For example, in Section~\ref{sec:split}, we present a technical result establishing that the composition 
of Avdonin maps, with slight adjustments, is again an Avdonin map.  This allows us to prove Theorem~\ref{thm:main1} for any $n$ by means of an iterative
procedure that splits off one interval at a time.  The details of this proof are not given explicitly but are subsumed in the proofs of Theorems~\ref{thm:main2} 
and \ref{thm:maincountable}.  In Section~\ref{sec:comb} we show that the union of the ranges of two Avdonin maps is in fact
the range of a third Avdonin map.  This allows for the combination of Riesz bases of exponentials for two intervals of given lengths into a Riesz basis for an interval
whose length is their sum.  These results are combined in Section~\ref{sec:proof} to prove Theorems~\ref{thm:main2} and \ref{thm:maincountable}.

\section{Related work}\label{sec:related}{}

Let us begin by discussing exponential bases 
$\mathcal E(\Lambda)$ for $\Lambda\subseteq\R$.  The following theorem is 
central to our analysis and follows trivially by means of rescaling from the fact
that $\mathcal E(\Z)$ forms an orthonormal basis for $L^2[0,1]$.

\begin{theorem}\label{thm:translateanddilate}
Given $\alpha\in\R$ and $a>0$, the set of exponentials $\mathcal E(\frac{\Z+\alpha}{a})$ is an orthogonal basis,
and hence a Riesz basis, for $L^2(I)$ for any interval $I$ with $|I|=a$.
\end{theorem}

There is a vast literature on the existence and properties of exponential basis
both from a mathematical perspective and from an engineering perspective in
the form of sampling of bandlimited functions (see Section~\ref{sec:sampling}).
Foundational texts and papers include \cite{Lev64, Lev96, Pav79, HruNikPav81, Hig96, Y80, Y01}.

While in general,  Riesz bases of exponentials can not be combined to form Riesz bases for larger intervals, the following result, well-attested in the literature though to our knowledge not stated in precisely this way, describes a setting where this is possible.
A proof is given in the appendix.

\begin{theorem}\label{thm:shiftedlattice1}
For $N\in\N$ and integer $0\le k\le N-1$, let $\Lambda_k=N\Z+k$.  Then the following hold.
\begin{itemize}
\item[{\rm (a)}] $\mathcal E(\Lambda_k)$ is a Riesz basis (in fact an orthogonal basis) for $L^2(I)$ for any interval $I$ with $|I|=\frac{1}{N}$.
\item[{\rm (b)}] For any $J\subseteq\{1,\,2,\,\dots,\,N\}$,
$$\mathcal\E\Big(\bigcup_{j\in J} \Lambda_j\Big)\ \mbox{is a Riesz basis for}\ L^2(I),$$
for any interval $I$ with $\displaystyle{|I|=\frac{|J|}{N}}$.
\end{itemize}
\end{theorem}

Exponential systems whose frequencies are unions of general shifted lattices
in $\R$ and $\R^d$ have been studied as far back as Kohlenberg in 1953 
\cite{Koh53}.  Such sets have been used to construct Riesz bases of 
exponentials for polygons in the plane (see \cite{LyuRas00, Wal17})
and for polytopes in higher dimensions \cite{DebLev20}.  They have also
been used to solve problems related to reconstruction from averages (e.g., 
\cite{Wal94, Wal96, Wal98, GroHeiWal00, Sei00}).

\begin{remark}\label{rem:counteregkadec}{\rm 
In general, it is not the case that given discrete disjoint sets 
$\Lambda_1$, $\Lambda_2\subseteq\R$ with the property
that $\mathcal E(\Lambda_i)$ is a Riesz basis for $L^2(I_i)$, that $\mathcal E(\Lambda_1\cup\Lambda_2)$ is a Riesz basis for intervals
$I$ with $|I|=|I_1|+|I_2|$.  For example, let
$$\Lambda_1=\Big\{2n-\frac14\Big\}_{n>0}\,\cup\,\Big\{2n+\frac14\Big\}_{n<0}\,\cup\,\{0\},\ \mbox{and}\
\Lambda_2=\Big\{2n+\frac34\Big\}_{n>0}\,\cup\,\Big\{2n-\frac34\Big\}_{n<0}.$$
Then $\mathcal E(\Lambda_1)$, $\mathcal{E}(\Lambda_2)$ are each a Riesz basis for any interval of length $1/2$,
but $\mathcal E(\Lambda_1\cup\Lambda_2)$ is not a Riesz basis for the unit interval.  Details
on this example can be found in \cite{Y80}, see also \cite{Y01, KN15}.  Indeed, this example shows that the $1/4$ in 
Kadec's theorem (\ref{eqn:kadectheorem}) is best possible.
}
\end{remark}

The following related result is due to Avdonin, in the context of so-called basis extraction and basis extension,
see \cite{AI95} Theorems~11.4.16 and 11.4.26.

\begin{theorem}\label{thm:extraction-extension}
Let $\Lambda\subseteq\R$ be such that $\mathcal E(\Lambda)$ is a Riesz basis for $L^2[0,T]$ for some $T>0$.
Then the following hold.
\begin{itemize}
\item[{\rm (a)}] For any $0<T'<T$, there exists $\Lambda'\subseteq\Lambda$ such that $\mathcal E(\Lambda')$ is a Riesz basis for $L^2[0,T']$.
\item[{\rm (b)}] For any $T'>T$, there exists $\Lambda'\supseteq\Lambda$ such that $\mathcal E(\Lambda')$ is a Riesz basis for $L^2[0,T']$.
\end{itemize}
\end{theorem}

What is not explicit in Theorem~\ref{thm:extraction-extension} (a) is whether the complement of an extracted basis is also a Riesz basis
for the complement of the interval.  However, if the basis is extracted from an orthonormal basis then this always holds.
A more general version of the following result appears in \cite{MM09} (Proposition 2.1) and a proof is 
provided in \cite{BCMS16} (Corollary 5.6).

\begin{theorem}\label{thm:HilbertComplement}
Let $S\subseteq[0,1]$ and suppose that for some $\Lambda\subseteq\Z$, $\mathcal{E}(\Lambda)$ is a Riesz basis for $L^2(S)$.
Then $\mathcal{E}(\Z\setminus\Lambda)$ is a Riesz basis for $L^2([0,1]\setminus S)$.
\end{theorem}

\begin{remark}\label{rem:countergedaegwan}
{\rm 
The orthogonality of the basis $\mathcal{E}(\Z)$ for $L^2[0,1]$ is essential in Theorem~\ref{thm:HilbertComplement}.  
The result fails if $\Z$ is replaced by even a small perturbation of $\Z$, as shown by the following example of Dae Gwan Lee \cite{L18}.  Let
$$\Lambda = \{2n\}_{n\in\N}\ \cup\ \{2n-1+\epsilon\}_{n>0}\ \cup\ \{2n+1-\epsilon\}_{n<0}.$$
Then $\Lambda$ is a perturbation of $\Z$ by no more than $\epsilon>0$, so that as long as $\epsilon<1/4$, (\ref{eqn:kadectheorem}) implies that
$\mathcal{E}(\Lambda)$ is a Riesz basis for $L^2[0,1]$.
Now taking
$$\Lambda_0 = \{2n\}_{n\le 0}\ \cup\ \{2n-1+\epsilon\}_{n>0},$$
it follows that $\Lambda_0$ is a perturbation of $\displaystyle{\Big\{2n - \Big(\frac{1}{2}-\frac{\epsilon}{2}\Big)\Big\}}$
by no more than $\displaystyle{\frac{1}{2}-\frac{\epsilon}{2}}$.  Therefore, by (\ref{eqn:kadectheorem}) with $a=1/2$, it follows that
$\mathcal{E}(\Lambda_0)$ is a Riesz basis for $L^2[0, 1/2]$.  

Now consider
$$\Lambda\setminus\Lambda_0 = \{2n\}_{n>0}\ \cup\ \{2n+1-\epsilon\}_{n<0}.$$
Note that the set $\Lambda\setminus\Lambda_0\,\cup\,\{0\}$ is a perturbation of $\displaystyle{\Big\{2n + \bigg(\frac{1}{2}-\frac{\epsilon}{2}\bigg)\Big\}}$,
and so by (\ref{eqn:kadectheorem}), $\mathcal{E}(\Lambda\setminus\Lambda_0\,\cup\{0\})$  forms a Riesz basis for $L^2[1/2,1]$.
Consequently, $\mathcal{E}(\Lambda\setminus\Lambda_0)$ is not complete and so does not form a Riesz basis.
}
\end{remark}

Theorem~\ref{thm:extraction-extension}(a) implies the following result of Seip \cite{S95}. 

\begin{theorem}\label{thm:seipresult}
For any $0 < \alpha < 1$, there exists a subset
$\Lambda\subseteq \Z$ such that $\mathcal E(\Lambda)$ is a Riesz basis for
$L^2[0,\alpha]$.
\end{theorem}

This result was used in \cite{S95} to construct exponential bases of the form 
$\E(\Lambda) = \{e^{2\pi i\lambda x}\colon \lambda\in\Lambda\}$,
$\Lambda\subseteq\Z$, for 
$L^2(I_1\,\cup\,I_2\,\cup\,\cdots\,\cup\, I_n)$,
where $\{I_j\}_{j=1}^n$ is a disjoint collection of subintervals
of $[0,1]$ whose lengths satisfy certain restrictions.

In a recent breakthrough result, Kozma and Nitzan \cite{KN15}
(see also \cite{KN16}) have completely solved this problem by showing that for any collection of subintervals $\{I_j\}_{j=1}^n$
of $[0,1]$, there exists $\Lambda\subset\Z$ such that $\E(\Lambda)$ is a Riesz basis for $L^2\Big(\bigcup_{j=1}^n I_j\Big)$
(for more on this problem, see e.g., \cite{Kat96, BezKat93, LyuSpi96, LyuSei97}
and references therein).
As a central part of the proof, the authors show that under some circumstances it is possible to ``combine
Riesz bases'' in the sense described above, but as the authors point out in the paper, their
construction does not necessarily give subsets $\Lambda_j$ such that $\Lambda = \bigcup_{j=1}^n \Lambda_j$ and
$\E(\Lambda_j)$ forms a Riesz basis for $L^2(I_j)$.  

Finally we note the following basis partition result of Lyubarskii and Seip 
\cite{LyuSei01}.

\begin{theorem}\label{thm:lyusei01}
Let $\E(\Lambda)$ be a Riesz basis of exponentials for $L^2[-1/2,1/2]$.  
For each $0<a<1$, there is a splitting
$$\Lambda = \Lambda'\,\cup\,\Lambda'',\ \Lambda'\,\cap\Lambda''=\emptyset$$
such that $\E(\Lambda')$ and $\E(\Lambda'')$ are Riesz bases for
$L^2[-a/2,a/2]$ and $L^2[-1/2(1-a),1/2(1-a)]$ respectively.
\end{theorem}

\section{Three fundamental tools}\label{sec:3results}

The results presented in this paper are based on the combination of three tools from three different mathematical disciplines. Before stating customized versions of these, we first define Riesz bases. The transition from orthonormal to Riesz bases is necessary as, for example, the $L^2[0,a]$ inner product of two exponentials with integer (or rational) frequencies never vanishes if $a$ is irrational. Hence,
obtaining an orthogonal basis $\mathcal E(\Lambda)$  with $\Lambda\subseteq \Z$ for $L^2[0,a]$ is not possible if $a$ is irrational, and we have to relax the customary concept of a basis of a Hilbert space, that is, of an orthonormal basis, to the concept of a Riesz basis \cite{OleC16}.

\begin{definition}\label{def:RB}
  The set of exponentials
   $\mathcal E(\Lambda)$, $\Lambda\subseteq\R$, is a {\it Riesz basis} of $L^2[a,b]$ if
   $\overline{\rm span}\,\mathcal E(\Lambda) {=} L^2[a,b]$ and for some $0{<}A,\!B{<}\infty$,
  \begin{align}\label{eqn:RB}
   A\sum_\lambda |c_\lambda|^2
   \leq \int_a^b \Big|\sum_{\lambda \in \Lambda} c_\lambda e^{2\pi i \lambda x}\Big|^2\, dx
  \leq B\sum_\lambda |c_\lambda|^2, \quad \{c_\lambda\}\in \ell^2(\Lambda).
  \end{align}
\end{definition}

To prove Theorems~\ref{thm:main1} -- \ref{thm:maincountable}, we shall not construct  subsets of the integers directly, but equivalently subsets of the integers shifted by $\frac 1 2$, that is, subsets of $\Z+\frac 1 2$.  
In the following, $
\lfloor x \rfloor$ denotes the largest element in  $\Z$ that is not larger than $x$.

\subsection{Avdonin's criterion for exponential Bases}
\label{sec:AvdoninMaps}

To show that $\mathcal E(\Lambda)$ with $\Lambda=\{\lambda_k\}\subseteq \Z$ is a Riesz bases for $L^2[0,a]$, we may consider the $\lambda_k$ as a perturbation of elements $\myfrac{k+\frac 1 2}a$ of a reference lattice and thereby $\mathcal E(\Lambda)$ as a perturbation of the orthogonal basis  $\mathcal E\Big(\myfrac{\Z+\frac 1 2}a\Big)$  of $L^2[0,a]$. For example,
Kadec's Theorem \cite{K64} states that if the elements of $\Lambda=\{\lambda_k\}\subset \R $ are close to the reference lattice $\myfrac{\Z+\frac 1 2}a$ in the sense that
\begin{equation}\label{eqn:kadectheorem}
  \sup_{k\in\Z} \Big|\lambda_k- \frac {k+\frac 1 2}a \Big|< \frac 1{4a},
\end{equation} then $\mathcal E(\Lambda)$ is a Riesz basis for $L^2([0,a])$ and therefore for any
interval $I$ of length $a$.

A substantial generalization of Kadec's theorem is Avdonin's theorem which plays a central role in our analysis.  Before stating it, we introduce the following vocabulary.

\begin{definition}\label{def:avdoninmap}
Let $\epsilon,\,a>0$ and $\alpha\in\R$.  An injective map
$$\varphi\colon\frac{\Z+\alpha}{a}\to\R$$ with separated range, that is, there exists $\delta>0$ such that  $|y-y'|>\delta$ for all $y\neq y'$ in the range of $\varphi$,
is an {\em $\epsilon$-Avdonin map} for the lattice $\frac{\Z+\alpha}{a}$ if there exists $R>0$ such that
\begin{equation}\label{eqn:avdoninmap}
\sup_{p\in\Z}\ \Big|\frac{1}{R}
  \sum_{\frac{k+\alpha}{a}\in[pR,(p+1)R)}
  \varphi\Big(\frac{k+\alpha}{a}\Big)
    -  \frac{k+\alpha}{a} \Big|
  < \epsilon.
  \end{equation}

If for some $R>0$ we can replace $<\epsilon$ by $=0$ on the right hand side
of \eqref{eqn:avdoninmap},then we speak of a $0$-Avdonin map and any 
$\varphi$ that is an $\epsilon$-Avdonin map for any $\epsilon>0$
\eqref{eqn:avdoninmap} is referred to as $\overrightarrow{0}$-Avdonin map.
\end{definition}

\begin{remark}\label{rem:avdoninmap}
  \rm
  If $\varphi$ is an $\epsilon$-Avdonin map for the lattice $\frac{\Z+\alpha}{a}$ satisfying
  \begin{align*}
  \sup_{k\in\Z}\ \Big|\varphi\Big(\frac{k+\alpha}{a}\Big)-\frac{k+\alpha}{a}\Big|\leq M
  \end{align*} for some $M\in\R$, then exists an $R_0$ such that \eqref{eqn:avdoninmap} holds for all $R>R_0$, a fact that will be repeatedly used later.
\end{remark}

Using Definition~\ref{def:avdoninmap}, Avdonin's theorem for regularly spaced frequency sets\footnote{Note that in \cite{A74}, 
Avdonin also considered general zero sets of so-called sine-type functions as reference sets.} can now be stated as follows \cite{A74}.

\begin{theorem}\label{thm:rieszbasis}
If  $\varphi$ is an $\frac{1}{4a}$-Avdonin map for the lattice $\frac{\Z+\alpha}{a}$, then
$\mathcal{E}\Big(\varphi\Big(\frac{\Z+\alpha}{a}\Big)\Big)$ is a Riesz basis for $L^2(I)$
for any interval $I$ with $|I|=a$.
\end{theorem}

Theorem~\ref{thm:rieszbasis} and its more general versions exist within
the context of a vast literature on the subject of exponential systems and 
their completeness and basis properties (for example, \cite{Beu86, 
BeuMal67, HruNikPav81, Lan67, Lev64, Lev96, Pav79}).

\subsection{Weyl-Khinchin equidistribution}

The Weyl equidistribution theorem asserts that the integer multiples of
 $\alpha$ irrational are equally distributed on the torus, that is,
\begin{align*}
\lim_{R\to\infty} \ \frac 1 R \ \sum_{k=1}^{R} \ k  \alpha \!\! \mod 1 \
 =\frac 1 2.
\end{align*}
Khinchin extended the result to  affine lattices; we shall employ the following customized generalization. 

\begin{theorem}\label{thm:WK}
For $a$ irrational and $\epsilon>0$ exists
$R_0$  so that for all $R\in\Z$ with $R>R_0$ and $m\in\Z$,
\begin{align*}
\Big|\frac{1}{R} \sum_{k=mR}^{(m+1)R-1}  &\frac {k+\frac 1 2} a \!\! \mod 1 \ \
 - \frac 1 2\Big| <\epsilon.
\end{align*}
\end{theorem}

A self-contained proof of a slight generalization of Theorem~\ref{thm:WK} is included in the appendix.

\subsection{Beatty-Fraenkel sequences}

In 1926, Beatty asked the readers of the American Mathematical Monthly to show that for positive irrational numbers $\alpha$ and $\beta$ with
$$\frac{1}{\alpha} + \frac{1}{\beta} = 1,$$
the sequences $\{\lfloor{n\alpha}\rfloor\}_{n\in\N}$ and $\{\lfloor{n\beta}\rfloor\}_{n\in\N}$ form a partition
of $\N$ \cite{B26}, \cite{BOHA27}.  Fraenkel  generalized this result 
and characterized for all $a\in\R$ all parameters $\gamma,\delta$ with the property that
$ \Big\lfloor \frac {\Z + \gamma}a \Big\rfloor$  and $ \Big\lfloor \frac {\Z + \delta}{1-a} \Big\rfloor$ partition $\Z$ \cite{F69}. Note that  the iterative procedure carried out in Section~\ref{sec:split} requires the choice of $\gamma=\frac 1 2 = \delta$.

\begin{theorem}  \label{thm:BF}
    The sequences $\Big\lfloor{\al{a}}\Big\rfloor  $ and $\Big\lfloor{\al{1-a}}\Big\rfloor $ partition $\Z$ for $0<a<1$ irrational.
  \end{theorem}

We provide a short proof, different from Fraenkel's approach, in Section~\ref{appendix:BF} in the appendix. 

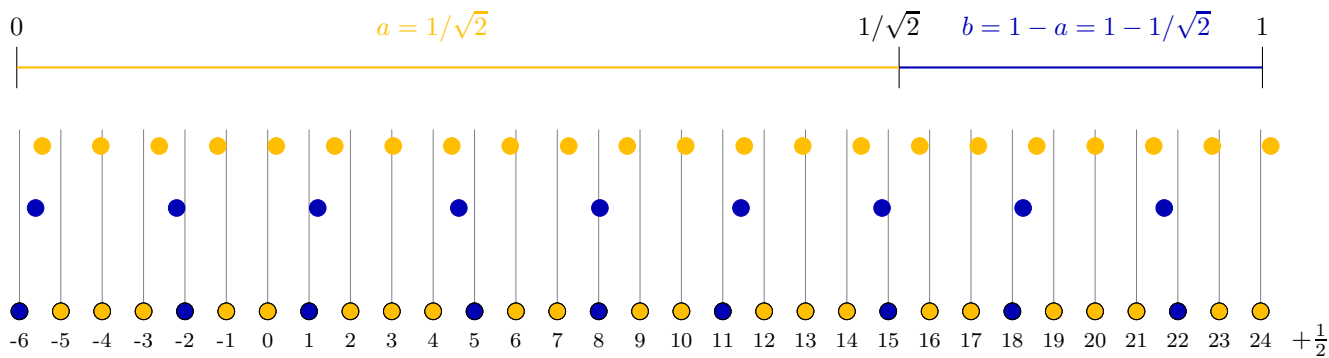
\begin{figure}[!h]

 \begin{tikzpicture}[scale=1.38]

\begin{scope}
\draw[-] (0,10) -- (12,10);
\draw [-] (0,9.8) -- (0,10.2);
\node at (0,10.4) {$0$};
\node[yellow] at (4,10.4) {$a=1/\sqrt{2}$};
\draw [-] (8.5,9.8) -- (8.5,10.2);
\node at (8.4,10.4) {$1/\sqrt{2}$};
\node[blue] at (10.3,10.4) {$b=1-a=1-1/\sqrt{2}$};
\draw [-] (12,9.8) -- (12,10.2);
\node at (12,10.4) {$1$};

\pause

  \draw[yellow,thick] [-] (0,10) -- (8.5,10);

\draw[blue,thick] [-] (8.5,10) -- (12,10);

\end{scope}
\end{tikzpicture}
\vspace{.5cm}

 \begin{tikzpicture}[scale=.55]

\foreach \x in {-6,-5,-4,-3,-2,-1,0,1,2,3,4,
      5,6,7,8,9,10,11,12,13,14,15,16,17,18,19,20,21,22,23,24}
 {
        \draw[gray,style=help lines] (\x,-6.5) -- (\x,-2.1);
        \draw[black,fill=white] (\x,-6.5) circle (.2cm);
        \node at (\x,-7.2) {\footnotesize \x};
      }
      \node at (25.2,-7.2) {$+\tfrac 1 2$};

\foreach \x in {-4*1.414,-3*1.414,-2*1.414,-1*1.414,0,
      1.414,2*1.414,3*1.414,4*1.414,5*1.414,6*1.414,7*1.414,8*1.414,9*1.414,10*1.414,11*1.414,12*1.414,13*1.414,14*1.414,15*1.414,16*1.414,17*1.414}
      {
        \draw[yellow,fill=yellow] (\x+0.707-.5,-2.5) circle (.2cm);
      }

% fourth lattice
      \foreach \x in {-2*3.41+1.71,-1*3.41+1.71,
      0+1.71,1*3.41+1.71,2*3.41+1.71,3*3.41+1.71,
      4*3.41+1.71,5*3.41+1.71,6*3.41+1.71}
      {
        \draw[blue,fill=blue] (\x-.5,-4) circle (.2cm);
      }

% fifth lattice = integers

 \foreach \x in {-6,-5,-4,-3,-2,-1,0,1,2,3,4,
      5,6,7,8,9,10,11,12,13,14,15,16,17,18,19,20,21,22,23,24}
 {
        \draw[black,fill=white] (\x,-6.5) circle (.2cm);
      }

      \foreach \x in {-5,-4,-3,-1,0,2,3,4,
      6,7,9,10,12,13,14,16,17,19,20,21,23,24}
 {
        \draw[black,fill=yellow] (\x,-6.5) circle (.2cm);
      }

      \foreach \x in {-6,-2,1,
      5,8,11,15,18,22}
 {
        \draw[black,fill=blue] (\x,-6.5) circle (.2cm);
      }

  \end{tikzpicture}

  \caption{Dividing the interval $[0,1]$ at $a=1/ {\sqrt{2}}$ and constructing Riesz bases for  $L^2[0,a]$ and $L^2[a,1]$ by rounding the yellow lattice
  $\frac {\Z+\frac 1 2 } {a}$  and the blue lattice $\frac {\Z+\frac 1 2} {1-a}$ to the lattice $\Z + \frac 1 2$. As guaranteed by Beatty-Fraenkel, the blue and yellow images partition $\Z + \frac 1 2$.}
\end{figure}

\subsection{Avdonin, Beatty-Fraenkel, and Weyl-Khinchin in concert}\label{sec:combiningABW}

The strategy of proof for Theorems~\ref{thm:main2} and~\ref{thm:maincountable} is to achieve a splitting of $[0,1]$ into $n$ intervals and $\Z$ into corresponding $n$ disjoint subsets, by splitting off one interval at a time.  The separation into two intervals is analyzed in detail in Section~\ref{sec:twointervals}, the iterative process in  Section~\ref{sec:split}.

Note that   different arguments are needed based on whether the lengths of the two intervals are rationally related or not. For illustration,  we consider the case where we split $[0,1]$ into the two intervals $[0,a]$ and $[a,1]$ with $\frac{1-a}a=\frac 1 a -1$ irrational and consequently $a$ and $1-a$ irrational.

Weyl-Khinchin, that is, Theorem~\ref{thm:WK}, implies that for  $a$ irrational and $\epsilon>0$, there exists
$R$  so that for all $m\in\Z$,
\begin{align}\label{eqn:WK}
\epsilon &> \Big|\frac{1}{R} \sum_{k=mR}^{(m+1)R-1}  \frac {k+\frac 1 2} a \!\! \mod 1 \ \
 - \frac 1 2\Big|  
 = \Big|\frac{1}{R} \sum_{k=mR}^{(m+1)R-1}
 \frac {k+\frac 1 2} a - \Big(\Big\lfloor \frac {k+\frac 1 2} a\Big\rfloor + \frac 1 2 \Big)\Big|\\
 &= \frac 1 a \Big|\frac{a}{R} \sum_{\frac {k+\frac 1 2} a\in 
 \left[\frac {mR+\frac 1 2 }{a},  \frac {(m+1)R+\frac 1 2 }{a}\right)} 
 \frac {k+\frac 1 2} a - \Big(\Big\lfloor \frac {k+\frac 1 2} a\Big\rfloor + \frac 1 2 \Big)\Big|.
 \end{align}
That is, the map 
\begin{align*}
\varphi_a: \frac {\Z+\frac 1 2}{a} \to {\Z+\frac 1 2}, \quad  \frac {k+\frac 1 2} a
\mapsto \Big\lfloor \frac {k+\frac 1 2} a\Big\rfloor + \frac 1 2 
\end{align*}
is an $a\epsilon$-Avdonin map for all $\epsilon>0$, that is, $\varphi_a$ is an $\overrightarrow 0$-Avdonin map. Replacing $a$ with $1-a$, we observe similarly that 

$\varphi_{1-a}$ is an $\overrightarrow 0$-Avdonin map from $\frac {\Z+\frac 1 2}{1-a}$ to $\Z+\frac 1 2$ for all $\epsilon>0$ as well, so that choosing $\epsilon<\min\Big\{\frac 1 {4a^2}, \frac 1 {4a(1-a)}\Big\}$, we can conlude that $\mathcal E(\lfloor \al{a}\rfloor +\frac 1 2 )$ is a Riesz basis for $L^2[0,a]$ and $\mathcal E(\lfloor \al{1-a}\rfloor+ \frac 1 2) $ is a Riesz basis for $L^2[a,1]$.  Theorem~\ref{thm:BF} now implies that the range of $\varphi_a$ and $\varphi_{1-a}$, 
that is, the sets  $\Big\lfloor \frac {\Z+\frac 1 2} {a}\Big\rfloor+ \frac 1 2$ and $\Big\lfloor \frac {\Z+\frac 1 2} {1-a}\Big\rfloor+ \frac 1 2$ partition $\Z+\frac 1 2$, and, hence, the sets $\Big\lfloor \frac {\Z+\frac 1 2} {a}\Big\rfloor $ and $\Big\lfloor \frac {\Z+\frac 1 2} {1-a}\Big\rfloor $ partition $\Z$.

\section{Splitting an interval into two}\label{sec:twointervals}

We begin the treatment of the problem at hand by  proving a version of Theorem~\ref{thm:main1} for $n=2$.

\begin{theorem}\label{thm:rounding}
For $0<c<d$ set \begin{align}\label{eqn:roundingfunction}
\varphi_c^{d}: \frac {\Z+\frac 1 2}{c} \to \frac {\Z+\frac 1 2} {d}, \quad  \frac {k+\frac 1 2} c
\mapsto \frac {\Big\lfloor  \frac {k+\frac 1 2} {\frac c{d}}\Big\rfloor + \frac 1 2}{d}.
\end{align}
For $a,b\in \R$ with $\frac a b$ irrational, or $\frac a b$ rational and $\frac{a}{a+b}=\frac{N_0}{K_0}$ in lowest terms with $K_0$ odd
set $\varphi=\varphi_a^{a+b}$ and $\psi=\varphi_b^{a+b}$, and in case of
$\frac{a}{a+b}=\frac{N_0}{K_0}$ in lowest terms with $K_0$ even set 
$\displaystyle \varphi=\varphi_a^{a+b} - \frac 1 {a+b}\chi_{\frac{K_0(2\Z-\frac 1 2)}{a+b}}$ and 
$\psi=\varphi_b^{a+b}-  \frac 1 {a+b}\chi_{\frac{K_0(2\Z+\frac 1 2)}{a+b}}$. Then 
\begin{itemize}
\item[{\rm (a)}]
$\Big\{\varphi\Big(\al{a}\Big),\,
       \psi\Big(\al{b}\Big)\Big\}$
forms a partition of $\al{a + b}$, and
\item[{\rm (b)}] for all $\epsilon>0$, the mappings $\varphi$ and $\psi$
are $\epsilon$-Avdonin maps for the lattices $\al{a}$ and $\al{b}$, respectively, and $0$-Avdonin maps if $\frac a b $ is rational.

\end{itemize}
\end{theorem}

Observe that the mapping $x\mapsto \lfloor{x}\rfloor +\frac12$ amounts to rounding $x$ to the element in $\Z+\frac12$ nearest to it,
where the ambiguity in case $x\in\Z$ is resolved by choosing the larger of the two choices, thereby creating a bias towards larger values. 

More generally for $\beta>0$, the mapping
$x\mapsto \frac{\lfloor{\beta x}\rfloor + \frac12}{\beta}$ maps $x$ to the element of $\frac{\Z+\frac12}{\beta}$ nearest to it. Applied to a lattice $\al{a}$, the unbalance created by the bias only appears if $\beta\frac {k+\frac 1 2}a\in\Z$ which is only possible if $\frac \beta a=\frac{K_0}{N_0}$ in lowest terms with $K_0$ even. This bias is correctd for by the adjustments to the rounding maps in Theorem~\ref{thm:rounding}.

Note that the Avdonin maps required in 
Theorem~\ref{thm:rounding} can be constructed locally.  These considerations are summarized in the following lemma.

\begin{lemma}\label{lem:localpartition}
Given $a,\,b>0$  there exists $K\in\N$ arbitrarily large such that for every $p\in\Z$ and $I=\big[\frac{pK}{a+b},\frac{(p+1)K}{a+b}\big)$,
\begin{equation}\label{eqn:localpartition}
\#\bigg(\frac{\Z+{\frac 1 2}}{a+b}\cap I\bigg) = \#\bigg(\frac{\Z+{\frac 1 2}}{a}\cap I\bigg) + \#\bigg(\frac{\Z+{\frac 1 2}}{b}\cap I\bigg)
\end{equation}
and the rounding map
\begin{align*}
  R: \bigg( \frac{\Z+\frac 1 2}{a} \cup \frac{\Z+{\frac 1 2}}{b} \bigg)\cap I\longrightarrow  \frac{\Z+{\frac 1 2}}{a+b} \cap I, \quad x \mapsto \frac{\lfloor (a+b) x\rfloor +\frac 1 2 }{a+b}
\end{align*} is well defined. In addition, it holds:
\begin{itemize}
\item[{\rm (a)}]  If $a/b$ is irrational we have $ \frac{\Z+{\frac 1 2}}{a} \cap  \frac{\Z+{\frac 1 2}}{b} =\emptyset$ and for any choice of $K\in \N$, 
(\ref{eqn:localpartition}) holds and the map $R$ is a bijection.

\item[{\rm (b)}]  If $a/b$ is rational with $\frac a {a+b}=\frac{N_0}{K_0}$ in lowest terms with $K_0$ odd, 
$ \frac{\Z+{\frac 1 2}}{a} \cap  \frac{\Z+{\frac 1 2}}{b} =\emptyset$ and for $K$ any integer multiple of $K_0$, (\ref{eqn:localpartition}) holds and $R$ is a bijection.

\item[{\rm (c)}] If $a/b$ is rational with $\frac a {a+b}=\frac{N_0}{K_0}$ in lowest terms with $K_0$ even, we have 
\begin{align*}
  \frac{\Z+{\frac 1 2}}{a} \cap  \frac{\Z+{\frac 1 2}}{b} = 
 \frac{K_0(\Z+\frac 1 2)}{a+b}=\frac{N_0(\Z+\frac 1 2)}{a}=\frac{(K_0-N_0)(\Z+\frac 1 2)}{b} 
\end{align*}
and for $K$ any integer multiple of $2K_0$, (\ref{eqn:localpartition}) holds and $R$ restricts to a bijection
$$\bigg( \bigg( \frac{\Z+{\frac 1 2}}{a} \setminus \frac{K_0(\Z+\frac 1 2)}{a+b}\bigg) \cup \frac{\Z+{\frac 1 2}}{b} \bigg)\cap I\quad \longrightarrow \quad \bigg( \frac{\Z+{\frac 1 2}}{a+b} \cap I\bigg) \setminus\frac{K_0(\Z+\frac 1 2)-\frac 1 2}{a+b}.$$
and acts as the identity on the additional points $\frac{N_0(\Z+\frac 1 2)}{a}=\frac{K_0(\Z+\frac 1 2)}{a+b}$. 
\end{itemize}
\end{lemma}

\begin{proof}
Without loss of generality, assume that $a+b=1$, for if not, letting $a'=\frac{a}{a+b}$ and $b'=\frac{b}{a+b}$ we have $a/b=a'/b'$ and $a'+b'=1$. 

 If $a/b$ is irrational, so is $a$ as $a+b=1$. Hence,  Theorem~\ref{thm:BF} implies that 
  $R$ maps bijectively
$\frac{\Z+{\frac 1 2}}{a}  \cup  \frac{\Z+{\frac 1 2}}{b} $ onto $\Z+\frac12$.  Further note that if $K_1,\,K_2\in\Z$,
then $x\in[K_1,K_2)$ if and only if $R(x)\in[K_1,K_2)$.
Hence it follows that for any $K\in\N$, $R$ maps
$\big(\frac{\Z+{\frac 1 2}}{a}\cap I\big)  \cup \big(\frac{\Z+{\frac 1 2}}{b}\cap I\big) $ bijectively onto $\big(\Z+\frac12\big) \cap I$.
Hence (a) and (\ref{eqn:localpartition}) hold.

Now assume that $a$ and therefore $b=1-a$ is rational and that $a=\frac{N_0}{K_0}$ in lowest terms.  First note that the sets 
$\frac{\Z+{\frac 1 2}}{a}$ and $\frac{\Z+{\frac 1 2}}{b}$ have period $K_0$ in the sense that for any $p\in\Z$,
$$\frac{\Z+{\frac 1 2}}{a} \cap [pK_0,(p+1)K_0) = \bigg(\frac{\Z+{\frac 1 2}}{a} \cap [0,K_0)\bigg) + pK_0$$
and similarly for $\frac{\Z+{\frac 1 2}}{b}$.  This follows immediately from the observation that if $n,\,p\in\Z$ then
$pK_0\le \frac{K_0}{N_0}(n+\frac12) < (p+1)K_0$ if and only if $0\le \frac{K_0}{N_0}((n-pN_0)+\frac12) < K_0$,
and similarly for $\frac{\Z+{\frac 1 2}}{b}$.  Hence in what follows we will consider only $k,\,j$ such that
$\frac{k+{\frac 1 2}}{a}\in[0,K_0)$ and $\frac{j+{\frac 1 2}}{b}\in[0,K_0)$ which occurs if and only if
$k=0,\,1,\,\dots,\,N_0-1$ and $j=0,\,1,\,\dots,\,K_0-N_0-1$.  From this it follows that (\ref{eqn:localpartition}) holds for 
$K$ any integer multiple of $K_0$ and $a/b$ rational.

Assume now that $K_0$ is odd.  
Let us assume, that for some $k,\,j\in\Z$, $R\Big( \frac{k+{\frac 1 2}}{a}\Big) = R\Big( \frac{j+{\frac 1 2}}{b}\Big)$, that is, that
\begin{align}\label{eqn:intersect} 
  \bigg\lfloor\frac{K_0(2k+1)}{2N_0}\bigg\rfloor = \bigg\lfloor\frac{K_0(2j+1)}{2(K_0-N_0)}\bigg\rfloor.
\end{align} 
Calling the common value $q$, we would have that $K_0(2k+1) = 2N_0 q + r$
and $K_0(2j+1) = 2(K_0-N_0) q + s$, where $0\le r\le 2N_0-1$, and $0\le s\le 2(K_0-N_0)-1$.  Since $K_0$ is odd, $r,\,s>0$.
Therefore, $2K_0(j+k+1-q) = r+s$ where $2\le r+s\le 2K_0-2$.  Hence $1\le K_0(j+k+1-q) \le K_0-1$ which is a contradiction.
From this it follows both that $ \frac{\Z+{\frac 1 2}}{a} \cap  \frac{\Z+{\frac 1 2}}{b} =\emptyset$
and that $R$ is injective.  That it is surjective follows from  (\ref{eqn:localpartition}), hence, (b) is shown.

Now suppose that $K_0$ is even in which case, $N_0$ must be odd.  In this case, if \eqref{eqn:intersect} holds with common value $q$,
the same
calculations as in the previous paragraph apply, and since $K_0$ is even, $r=0$ and $s=0$ are possible.
Therefore we have  $0\le K_0(j+k+1-q) \le K_0-1$ which can only hold if $j+k+1=q$ so that $r=s=0$.
This gives immediately that 
$$\frac{K_0(2k+1)}{2N_0} = \frac{K_0(2j+1)}{2(K_0-N_0)}$$
and that each of these fractions is an integer.  Given that $0\le k\le N_0-1$ and $0\le j\le K_0-N_0-1$ it follows that
$k=\frac{N_0-1}{2}$, $j=\frac{K_0-N_0-1}{2}$, and that 
$$\frac{K_0(2k+1)}{2N_0} = \frac{K_0(2j+1)}{2(K_0-N_0)} = \frac{K_0}{2}.$$
From this and periodicity considerations, it follows that $\frac{\Z+{\frac 1 2}}{a} \cap \frac{\Z+{\frac 1 2}}{b} = K_0(\Z+1/2)$.
It also follows from the fact that $K_0/N_0>1$ that 
$$R\Big(\frac{\Z+{\frac 1 2}}{a}\cup \frac{\Z+{\frac 1 2}}{b}\Big) \cap \big(K_0(\Z+1/2)-1/2\big) = \emptyset,$$
and (c) follows.
\end{proof}

\noindent{\it Proof of Theorem~\ref{thm:rounding} in the rational case.}\quad
In the case that $\frac a b$ is irrational we refer to our discussion in Section~\ref{sec:combiningABW}. 
Note that the choice of $\varphi$ and $\psi$ does not depend on $\epsilon$. 

If $\frac{a} {b}$ is rational, this means that both $\frac{a}{a+b}$ and $\frac{b}{a+b}$ are rational.  
If $\frac{a}{a+b}=\frac{N_0}{K_0}$ in lowest terms with $K_0$ even, Lemma~\ref{lem:localpartition} implies that 
\begin{align}
  \varphi_a^{a+b}\Big(\al{a}\Big)\cap 
       \varphi_b^{a+b}\Big(\al{b}\Big)=  \frac{K_0(\Z+\frac 1 2)}{a+b}, \quad \varphi_a^{a+b}\Big(\al{a}\Big)\cup 
       \varphi_b^{a+b}\Big(\al{b}\Big)=\al{a + b} \setminus \frac{K_0(\Z-\frac 1 2)}{a+b}
\end{align}
which necessitates the correction through subtraction of characteristic functions in order to obtain a partition of $\al{a+b}$.
If $K_0$ is odd, then no such adjustment needs to be made.  The full proof in the rational case follows.

Without loss of generality, we can assume as in Lemma~\ref{lem:localpartition} that $a+b=1$ with $a$ and therefore $b$ rational, 
specifically that $a=\frac{N_0}{K_0}$ and $b=\frac{K_0-N_0}{K_0}$
in lowest terms.  As mentioned before, the sets $\frac{\Z+{\frac 1 2}}{a}$ and $\frac{\Z+{\frac 1 2}}{b}$ have period $K_0$, so it suffices to restrict our attention to the interval $[0,K_0)$.
Note that $\frac{k+{\frac 1 2}}{a}\in[0,K_0)$ if and only if $k=0,\,1,\,\dots,\,N_0-1$
and that $\frac{j+{\frac 1 2}}{b}\in[0,K_0)$ if and only if $j=0,\,1,\,\dots,\,K_0-N_0-1$.

Let us first consider the case that  $K_0$ is odd.  With $\varphi$ and $\psi$ as in Theorem~\ref{thm:rounding}, 
Lemma~\ref{lem:localpartition} implies (a) in Theorem~\ref{thm:rounding}. 

We shall now focus on the map $\varphi$ as (b) in Theorem~\ref{thm:rounding} follows for $\psi$ in the same manner. 

Consider for  $k=0,\,1,\,\dots,\,N_0-1$ the term
\begin{align*}
  \frac{k+\frac12}{a}= \frac{K_0(2k+1)}{2N_0}=\Big\lfloor\frac{K_0(2k+1)}{2N_0} \Big\rfloor+\frac {r_k}{2N_0}
\end{align*}
with $r_k\in \{1,\ldots, 2N_0-1\}$ since $K_0$ being odd implies $\frac{K_0(2k+1)}{2N_0}\notin \Z$. 
Further,
\begin{align*}
  \frac{N_0-1-k+\frac12}{a}&= \frac{K_0(2(N_0-1-k)+1)}{2N_0}= K_0- 
  \frac{K_0(2 k+1)}{2N_0}\\
  &= K_0- \Big\lfloor\frac{K_0(2k+1)}{2N_0} \Big\rfloor-\frac {r_k}{2N_0}
  = K_0 + \Big\lfloor-\frac{K_0(2k+1)}{2N_0} \Big\rfloor+1 -\frac {r_k}{2N_0}\\
  &= \Big\lfloor \frac{N_0-1-k+\frac12}{a} \Big\rfloor+ \frac {2N_0-r_k}{2N_0}
\end{align*}
where we used that $-\lfloor x \rfloor = \lfloor -x \rfloor+1$ for $x>0$ and $x\notin \Z$. Consequently
\begin{align}\label{eqn:summingrational}
   \sum_{k=0}^{N_0-1} & \bigg\lfloor{\frac{K_0(2k+1)}{2N_0}}\bigg\rfloor
        -  \frac{K_0(2k+1)}{2N_0}  
\\ \notag &  =    \frac12\left[\sum_{k=0}^{N_0-1} \bigg\lfloor{\frac{K_0(2k+1)}{2N_0}}\bigg\rfloor
        - \frac{K_0(kn+1)}{2N_0}  + \sum_{k=0}^{N_0-1} \bigg\lfloor{\frac{K_0(2(N_0-1-k)+1)}{2N_0}}\bigg\rfloor
        -  \frac{K_0(2(N_0-1-k)+1)}{2N_0} \right] \\
& \notag =    \frac12 \bigg[\sum_{k=0}^{N_0-1} -\frac{r_k}{2N_0}
        +  \sum_{k=0}^{N_0-1} \frac{r_k-2N_0}{2N_0}\bigg] = -\frac{N_0}{2}
\end{align}
and
\begin{align*}
 \sum_{\frac{k+{\frac 1 2}}{a}\in[0,K_0)} \varphi\Big(\frac{k+{\frac 1 2}}{a}\Big)-\frac{k+{\frac 1 2}}{a} =
   \sum_{k=0}^{N_0-1} & \bigg\lfloor{\frac{K_0(2k+1)}{2N_0}}\bigg\rfloor+\frac 1 2
        -  \frac{K_0(2k+1)}{2N_0}   =\frac{N_0}{2} -\frac{N_0}{2}=0.
\end{align*}
Since for $N$ any integer multiple of $K_0$
$$\sum_{\frac{k+{\frac 1 2}}{a}\in[pN,(p+1)N)} \varphi\Big(\frac{k+{\frac 1 2}}{a}\Big)-\frac{k+{\frac 1 2}}{a} = 0,$$
$\varphi$ is an $0$-Avdonin map.

Since the sums
$$\sum_{\frac{k+{\frac 1 2}}{a}\in[pN,(p+1)N)}
\Big|\varphi\Big(\as{k}{a}\Big)-\as{k}{a}\Big| \le N_0,$$
for all $p\in\Z$, it follows from Remark~\ref{rem:avdoninmap} that Theorem~\ref{thm:rounding}(b) holds.

Now assume that $K_0$ is even, which implies that $N_0$ is odd.  Recall that in this case,
the sequences
$\lfloor{\frac{n+{\frac 1 2}}{a}}\rfloor$ and
$\lfloor{\frac{n+{\frac 1 2}}{b}}\rfloor$,
for $n\in\Z$ do not partition $\Z+\frac 1 2$, in fact, the values
$ {K_0(\Z+\frac 1 2)}$ are assumed by both sequences while the values ${K_0(\Z+\frac 1 2)-\frac 1 2} $ are assumed by neither.

For $k=0,\,\dots,\,N_0-1$ and $j=0,\,\dots,\,K_0-N_0-1$,
calculations carried out in the proof of Lemma~\ref{lem:localpartition} show that
$\lfloor{\frac{k+{\frac 1 2}}{a}}\rfloor = \lfloor{\frac{j+{\frac 1 2}}{b}}\rfloor$
if and only if $\frac{k+{\frac 1 2}}{a} = \frac{j+{\frac 1 2}}{b}$ which occurs only when
$k=\frac{N_0-1}{2}$ and $j=\frac{K_0-N_0-1}{2}$ and that in this case,
$\frac{k+{\frac 1 2}}{a} = \frac{j+{\frac 1 2}}{b} = \frac{K_0}{2}$.  It also holds that
$\frac{k+{\frac 1 2}}{a} = \frac{K_0(2k+1)}{2N_0}$ is an integer only when $k=\frac{N_0-1}{2}$
and $\frac{j+{\frac 1 2}}{b} = \frac{K_0(2k+1)}{2(K_0-N_0)}$ is an integer only when
$j=\frac{K_0-N_0-1}{2}$.
Therefore, only the $\frac {N_0-1} 2$ summand in \eqref{eqn:summingrational} is 0 and we obtain  
\begin{eqnarray*}
&     &  \sum_{k=0}^{N_0-1} \bigg\lfloor{\frac{K_0(2k+1)}{2N_0}}\bigg\rfloor
        - \bigg(\frac{K_0(2k+1)}{2N_0}\bigg) \\
&  =  &  \sum_{k=0}^{\frac{N_0-1}{2}-1} \bigg\lfloor{\frac{K_0(2k+1)}{2N_0}}\bigg\rfloor
        - \bigg(\frac{K_0(2k+1)}{2N_0}\bigg) \\
&     &  \hskip.25in
        +  \sum_{k=0}^{\frac{N_0-1}{2}-1} \bigg\lfloor{\frac{K_0(2(N_0-1-k)+1)}{2N_0}}\bigg\rfloor
        - \bigg(\frac{K_0(2(N_0-1-k)+1)}{2N_0}\bigg) \\
&  =  &  \sum_{k=0}^{\frac{N_0-1}{2}-1} -\frac{r_k}{2N_0}
        +  \sum_{k=0}^{\frac{N_0-1}{2}-1} \frac{r_k-2N_0}{2N_0} = -\bigg(\frac{N_0-1}{2}\bigg)
\end{eqnarray*}
and similarly
$$\sum_{j=0}^{N_0-1} \bigg\lfloor{\frac{K_0(2j+1)}{2(K_0-N_0)}}\bigg\rfloor
        - \bigg(\frac{K_0(2j+1)}{2(K_0-N_0)}\bigg) = -\bigg(\frac{K_0-N_0-1}{2}\bigg).$$
On $[0,2K_0)$, $\varphi$ and $\psi$ satisfy
\begin{equation}\label{eqn:definerationalvarphi}
\varphi\Big(\as{k}{a}\Big) = \left\{
\begin{array}{ll}
\Big\lfloor{\as{k}{a}}\Big\rfloor + \frac12 & \mbox{if $\as{k}{a}\notin\Z$} \\
                                         &                            \\
\Big\lfloor{\as{k}{a}}\Big\rfloor + \frac12 & \mbox{if $\as{k}{a}\in\Z\cap [0,K_0)$} \\
                                         &                                     \\
\Big\lfloor{\as{k}{a}}\Big\rfloor - \frac12 & \mbox{if $\as{k}{a}\in\Z\cap [K_0,2K_0)$}
\end{array}\right.
\end{equation}
and
\begin{equation}\label{eqn:definerationalpsi}
\psi\Big(\as{k}{a}\Big) = \left\{
\begin{array}{ll}
\Big\lfloor{\as{k}{b}}\Big\rfloor + \frac12 & \mbox{if $\as{k}{b}\notin\Z$} \\
                                         &                            \\
\Big\lfloor{\as{k}{b}}\Big\rfloor - \frac12 & \mbox{if $\as{k}{b}\in\Z\cap [0,K_0)$} \\
                                         &                            \\
\Big\lfloor{\as{k}{b}}\Big\rfloor + \frac12 & \mbox{if $\as{k}{b}\in\Z\cap [K_0,2K_0)$}
\end{array}\right.
\end{equation}
and it follows that
\begin{align*}
\sum_{k=0}^{N_0-1}
         \varphi\bigg(\frac{K_0(2k+1)}{2N_0}\bigg) - \bigg(\frac{K_0(2k+1)}{2N_0}\bigg) 
&=    \sum_{k\ne\frac{N_0-1}{2}}
         \varphi\bigg(\frac{K_0(2k+1)}{2N_0}\bigg) - \bigg(\frac{K_0(2k+1)}{2N_0}\bigg) + \frac12 \\
  &=    -\bigg(\frac{N_0-1}{2}\bigg) + \frac{1}{2}(N_0-1) + \frac12 = \frac12,
\end{align*}
and similarly that
$$\sum_{k=N_0}^{2N_0-1}
         \varphi\bigg(\frac{K_0(2k+1)}{2N_0}\bigg) - \bigg(\frac{K_0(2k+1)}{2N_0}\bigg) = -\frac12.$$
         
Hence
$$\sum_{\frac{k+{\frac 1 2}}{a}\in[0,2K_0)} \varphi\Big(\frac{k+{\frac 1 2}}{a}\Big)-\frac{k+{\frac 1 2}}{a} = 
  \sum_{k=0}^{2N_0-1}
         \varphi\bigg(\frac{K_0(2k+1)}{2N_0}\bigg) - \bigg(\frac{K_0(2k+1)}{2N_0}\bigg) = \frac12 - \frac12 = 0.$$
The same equation holds for $\psi$ with $N_0$ replaced by $K_0-N_0$ in the above.
Letting $N$ be a multiple of $2K_0$, we observe that we obtain $0$-Avdonin maps.
Moreover, again by Remark~\ref{rem:avdoninmap}, 
Theorem~\ref{thm:rounding}(b) holds.

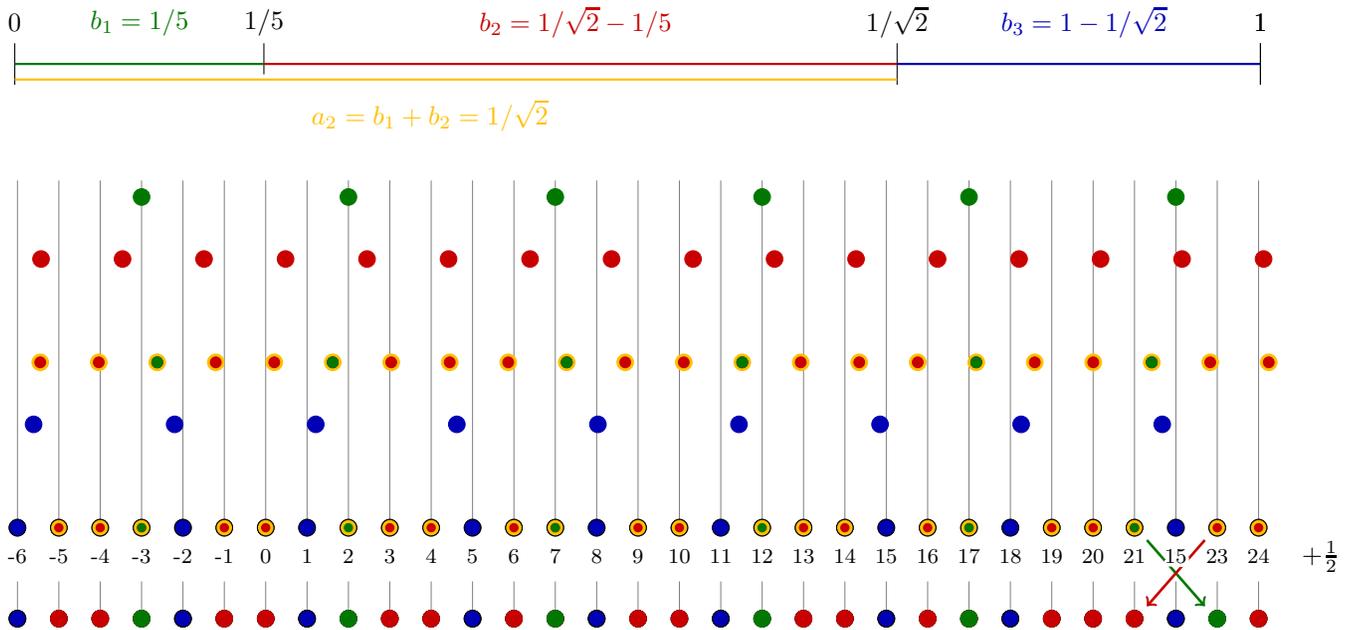
\begin{figure}
%\centering
\begin{tikzpicture}[scale=1.38]

\begin{scope}
\draw[-] (0,10) -- (12,10);
\draw [-] (0,9.8) -- (0,10.2);
\node at (0,10.4) {$0$};
\draw [-] (2.4,9.9) -- (2.4,10.2);
\node at (2.4,10.4) {$1/5$};
\draw [-] (8.5,9.8) -- (8.5,10.2);
\node at (8.5,10.4) {$1/\sqrt{2}$};
\draw [-] (12,9.8) -- (12,10.2);
\node at (12,10.4) {$1$};

  \draw[green,thick] [-] (0,10) -- (2.4,10);
  \node[green] at (1.2,10.4) {$b_1=1/5$};

  \draw[red,thick] [-] (2.4,10) -- (8.5,10);
  \node[red] at (5.4,10.4) {$b_2=1/\sqrt{2}-1/5$};

  %\node[red] at (5.4,9.4) {$\frac {\Z+\frac 1 2 } {b_2}\approx1.97\Z+0.99$};

  \draw[blue,thick] [-] (8.5,10) -- (12,10);
  \node[blue] at (10.3,10.4) {$b_3=1-1/\sqrt{2}$};

  %\node[blue] at (10.3,9.4) {$\frac {\Z+\frac 1 2 } {b_3} \approx 3.41\Z+1.71$};

  \draw [-] (12,9.8) -- (12,10.2);
  \node at (12,10.4) {$1$};

  \draw[yellow,thick] [-] (0,9.85) -- (8.5,9.85);
 \node[yellow] at (4,9.5) {$a_2=b_1+b_2=1/\sqrt{2}$};

  %\node[yellow] at (4,11) {$\frac {\Z+\frac 1 2} {b_1+b_2}
  %\frac{5\sqrt{2}}{5-\sqrt{2}}\Z
  %=\sqrt{2}\Z+ \frac 1 {\sqrt{2}}$};

\end{scope}
\end{tikzpicture}

\vspace{.5cm}
 \begin{tikzpicture}[scale=.55]

  % integers and grid

\foreach \x in {-6,-5,-4,-3,-2,-1,0,1,2,3,4,
      5,6,7,8,9,10,11,12,13,14,15,16,17,18,19,20,21,22,23,24}
 {
        \draw[gray,style=help lines] (\x,-6.5) -- (\x,1.9);
        \draw[black,fill=white] (\x,-6.5) circle (.2cm);
        \node at (\x,-7.2) {\footnotesize \x};
      }

     \node at (25.5,-7.2) {$+\tfrac 1 2$};

%green lattice

    \foreach \x in {-3,2,7,12,17,22}
      {
        \draw[green,fill=green] (\x,1.5) circle (.2cm);
      }

% red lattice

       \foreach \x in {-3*1.97+0.98,-2*1.97+0.98,-1*1.97+0.98,0+0.98,1*1.97+0.98,2*1.97+0.98,3*1.97+0.98,4*1.97+0.98,5*1.97+0.98,6*1.97+0.98,7*1.97+0.98,8*1.97+0.98,
       9*1.97+0.98,10*1.97+0.98,11*1.97+0.98,12*1.97+0.98}
      {
        \draw[red,fill=red] (\x-.5,0) circle (.2cm);
      }

% blue lattice

 \foreach \x in {-2*3.41+1.71,-1*3.41+1.71,
      0+1.71,1*3.41+1.71,2*3.41+1.71,3*3.41+1.71,
      4*3.41+1.71,5*3.41+1.71,6*3.41+1.71}
      {
        \draw[blue,fill=blue] (\x-.5,-4) circle (.2cm);
      }

% draw yellow interval and mark lattice on previous

% yellow lattice

\foreach \x in {-4*1.414,-3*1.414,-2*1.414,-1*1.414,0,
      1.414,2*1.414,3*1.414,4*1.414,5*1.414,6*1.414,7*1.414,8*1.414,9*1.414,10*1.414,11*1.414,12*1.414,13*1.414,14*1.414,15*1.414,16*1.414,17*1.414}
      {
        \draw[yellow,fill=yellow] (\x+0.707-.5,-2.5) circle (.2cm);
      }

% yellow and blue in integers

      \foreach \x in {-5,-4,-3,-1,0,2,3,4,
      6,7,9,10,12,13,14,16,17,19,20,21,23,24}
 {
        \draw[black,fill=yellow] (\x,-6.5) circle (.2cm);
      }

      \foreach \x in {-6,-2,1,
      5,8,11,15,18,22}
 {
        \draw[black,fill=blue] (\x,-6.5) circle (.2cm);
      }

    % green lattice

% green and red on yellow lattice

      \foreach \x in {-2*1.414,1.414,
      5*1.414,8*1.414,12*1.414,15*1.414}
      {
        \draw[green,fill=green] (\x+0.707-.5,-2.5) circle (.13cm);
      }

      \foreach \x in {-4*1.414,-3*1.414,-2*1.414,-1*1.414,0,
      1.414,2*1.414,3*1.414,4*1.414,5*1.414,6*1.414,7*1.414,8*1.414,9*1.414,10*1.414,11*1.414,12*1.414,13*1.414,14*1.414,15*1.414,16*1.414,17*1.414}
      {
        \draw[red,fill=red] (\x+0.707-.5,-2.5) circle (.13cm);
      }
      \foreach \x in {-2*1.414,1.414,
      5*1.414,8*1.414,12*1.414,15*1.414}
      {
        \draw[green,fill=green] (\x+0.707-.5,-2.5) circle (.13cm);
      }

      \foreach \x in {-5,-4,-3,-1,0,3,4,
      6,7,9,10,13,14,16,17,19,20,21,23,24}
 {
        \draw[red,fill=red] (\x,-6.5) circle (.1cm);
      }

      \foreach \x in {-3,2,7,12,17,21}
 {
        \draw[green,fill=green] (\x,-6.5) circle (.10cm);
      }

% interchange

\foreach \x in {-6,-5,-4,-3,-2,-1,0,1,2,3,4,
      5,6,7,8,9,10,11,12,13,14,15,16,17,18,19,20,21,22,23,24}
 {
        \draw[gray,style=help lines] (\x,-8.7) -- (\x,-7.8);
        \draw[black,fill=white] (\x,-8.7) circle (.2cm);
        %\node at (\x,-8.7) {\footnotesize \x};
      }

      \foreach \x in {-6,-5,-4,-3,-2,-1,0,1,2,3,4,
      5,6,7,8,9,10,11,12,13,14,15,16,17,18,19,20,21,22,23,24}
 {
        \draw[black,fill=yellow] (\x,-8.7) circle (.2cm);
      }

      \foreach \x in {-6,-2,1,
      5,8,11,15,18,22}
 {
        \draw[black,fill=blue] (\x,-8.7) circle (.2cm);
      }
   \foreach \x in {-5,-4,-3,-1,0,3,4,
      6,7,9,10,13,14,16,17,19,20,21,24}
 {
        \draw[red,fill=red] (\x,-8.7) circle (.2cm);
      }

      \foreach \x in {-3,2,7,12,17,23}
 {
        \draw[green,fill=green] (\x,-8.7) circle (.2cm);
      }

     \draw[green,line width=0.35mm] [->]   (14.3+7,-6.8) -- (15.7+7,-8.4);
     \draw[red,line width=0.35mm] [->]   (15.7+7,-6.8) -- (14.3+7,-8.4);
     \node at (15+7,-7.2) {\contour{white}{\footnotesize 15}};

  \end{tikzpicture} \caption{Dividing the interval $[0,1)$ at $a_1=\frac 1 5$ and at $a_2=\frac 1 {\sqrt{2}}$ to obtain three intervals, of length $b_1=1/5$, $b_2= 1 /{\sqrt{2}}-  1/5$ and $b_3=1-1/\sqrt{2}$. The yellow lattice $\sqrt{2}\Z+ 1 /\sqrt{2}$ and the blue  lattice $(\Z+ 1/2)/(1-1/\sqrt{2})\approx 3.41\Z+1.71$ are first rounded to the half integers $\Z + \frac 1 2$. As the yellow interval is then split into the green and red intervals, the respective lattices $5\Z+5/2$ and $(\Z+ 1/2)/(1/\sqrt{2}-1/5)\approx 1.97\Z+0.99$ are then mapped to the closest point in the yellow lattice  $\sqrt{2}\Z+ 1 /\sqrt{2}$, and then transported to the target lattice $\Z +  1 /2$. The successive application of a rounding function may not lead to a partition of $\Z+1/2$ into
   Riesz bases for  $[0,a_1)$, $[a_1,a_2)$ and  $[a_2,1)$, consequently an adjustment to the assignment, that is, a switching of red and green points in the target may have to be applied.}

\end{figure}

\section{Successively splitting intervals}\label{sec:split}

This section contains the central argument of this paper. We show, that, aside of required permittable small adjustments, the composition of   Avdonin maps are again Avdonin maps. The following will allow us to carry out an inductive argument based on the induction anchor given in Theorem~\ref{thm:rounding}.

\begin{theorem}\label{thm:rearrangement}
Suppose that for some $a,\,b,\,\epsilon,\,\delta>0$, 
injective maps
$$\widehat{\varphi}\colon \al{a}  \to \al{a+b},\qquad
  \widehat{\psi}\colon \al{b} \to \al{a+b},\qquad{\mbox and}\qquad
  \sigma\colon \al{a+b} \to \Z+\tfrac 12$$
have the property that there exist $K\in\N$ arbitrarily large such that for any $p\in\Z$, and $I=[\frac{pK}{a+b},\frac{(p+1)K}{a+b})$,
$\widehat{\varphi}$ (resp. $\widehat{\psi}$) maps $\al{a} \cap I$ (resp. $\al{b}\cap I$) into $\al{a+b}\cap I$, and satisfies
\begin{itemize}
\item[\rm{(1)}]
$\displaystyle{\Big\{\widehat{\varphi}\Big(\al{a}\Big),\,
\widehat{\psi}\Big(\al{b} \Big)\Big\}}$
forms a partition of $\al{a+b}$,
\item[\rm{(2)}]  there exists $\widehat{M}>0$ such that for all $k,m,\ell\in\Z$,
$$
\Big|\widehat{\varphi}\Big(\as{k}{a}\Big)-\as{k}{a}\Big|,\,
\Big|\widehat{\psi}\Big(\as{\ell}{b}\Big)-\as{\ell}{b}\Big|,\,
\Big|\sigma\Big(\as{m}{a+b}\Big)-\as{m}{a+b}\Big|
\le \widehat{M}
$$
and
\item[\rm{(3)}]  the mappings $\widehat \varphi$ and $\widehat \psi$ are $\delta$-Avdonin maps,
and $\sigma$ is an $\epsilon$-Avdonin map.
\end{itemize}
Then there exist injective maps
$$\varphi\colon \al{a}  \to \al{a+b},\qquad
  \psi\colon \al{b} \to \al{a+b},$$
such that
\begin{itemize}
\item[\rm{(a)}]  $\displaystyle{\Big\{\varphi\Big(\al{a}\Big),\,
\psi\Big(\al{b}\Big)\Big\}}$ forms a partition of
$\al{a+b}$,
\item[\rm{(b)}]  there exists $M>0$ such that for all $j\in\Z$,
$$\Big|\sigma\circ\varphi\Big(\frac{j+\frac12}{a}\Big)-\frac{j+\frac12}{a}\Big|,\,
\Big|\sigma\circ\psi\Big(\frac{j+\frac12}{b}\Big)-\frac{j+\frac12}{b}\Big|
\le M,
$$
and
\item[\rm{(c)}]  the mappings
$$\sigma\circ\varphi\colon \al{a}\to\Z+\frac12,\qquad\mbox{and}\qquad
  \sigma\circ\psi\colon \al{b}\to\Z+\frac12$$
are $(\epsilon+3\delta)$-Avdonin maps.

\end{itemize}
\end{theorem}

\begin{proof}
Fix an integer $K_0$ (to be chosen later) and consider an arbitrary interval of the form
$I =[\frac{\ell_0K_0}{a+b},\frac{(\ell_0+1)K_0}{a+b})$ for some $\ell_0\in\Z$.
Clearly there are $K_0$ elements of $\frac{\Z+\frac12}{a+b}$ in $I$ and
Lemma~\ref{lem:localpartition}
implies that if $N_0=\#\{\frac{\Z+\frac12}{a}\cap I\}$ then
$K_0-N_0=\#\{\frac{\Z+\frac12}{b}\cap I\}$.   Let
$$\frac{\Z+\frac12}{a}\cap I
  = \Big\{\frac{k_0+\frac12}{a},\,\frac{k_0+\frac32}{a},\,\dots,\,\frac{k_0+N_0-1+\frac12}{a}\Big\}$$
and
$$\frac{\Z+\frac12}{b}\cap I
  = \Big\{\frac{m_0+\frac12}{b},\,\frac{m_0+\frac32}{b},\,\dots,\,\frac{m_0+K_0-N_0-1+\frac12}{b}\Big\}.$$
Given an injective map
\begin{equation}\label{eqn:injectivemap}
\varphi\colon \frac{\Z+\frac12}{b}\cap I \to \frac{\Z+\frac12}{a+b}\cap I
\end{equation}
define
$$S(\varphi) = \sum_{k=k_0}^{k_0+N_0-1}
  \sigma\circ\varphi\Big(\frac{k+\frac12}{a}\Big) - \Big(\frac{k+\frac12}{a}\Big).$$
We wish to consider values of $S(\varphi)$ as $\varphi$ runs through all injections of the form
(\ref{eqn:injectivemap}).
Note that $S(\varphi)$ depends only on the range of $\varphi$ and not on $\varphi$ itself.

We first show that there exist injections $\varphi$ such that $S(\varphi)<0$.  To see this, choose
some integer $M_0$ with $0\le M_0\le N_0$ and define $\varphi$ so that
\begin{eqnarray*}
\varphi\Big(\frac{\Z+\frac12}{a}\,\cap\, I\Big)
&  =  & \widehat{\varphi}\Big(\Big\{\frac{k_0+\frac12}{a},\,\dots,\,
        \frac{k_0+N_0-M_0-1+\frac12}{a}\Big\}\Big) \\
&     & \cup\,\,\,\,
\widehat{\psi}\Big(\Big\{\frac{m_0+\frac12}{b},\,\dots,\,\frac{m_0+M_0-1+\frac12}{b}\Big\}\Big)
\end{eqnarray*}
Therefore,
\begin{align*}
S(\varphi)  =& \sum_{k=k_0}^{k_0+N_0-1}
         \sigma\circ\varphi\Big(\frac{k+\frac12}{a}\Big) - \Big(\frac{k+\frac12}{a}\Big)\\
=&  \sum_{k=k_0}^{k_0+N_0-M_0-1}
         \sigma\circ\widehat{\varphi}\Big(\frac{k+\frac12}{a}\Big) - \Big(\frac{k+\frac12}{a}\Big) +
       \sum_{m=m_0}^{m_0+M_0-1}
         \sigma\circ\widehat{\psi}\Big(\frac{m+\frac12}{b}\Big) - \Big(\frac{m+\frac12}{b}\Big)\\
   &+    \sum_{m=m_0}^{m_0+M_0-1} \Big(\frac{m+\frac12}{b}\Big) -
         \sum_{k=k_0+N_0=M_0}^{k_0+N_0-1} \Big(\frac{k+\frac12}{a}\Big) \\
= &    \sum_{k=k_0}^{k_0+N_0-M_0-1} \sigma\circ\widehat{\varphi}\Big(\frac{k+\frac12}{a}\Big) - \widehat{\varphi}\Big(\frac{k+\frac12}{a}\Big)
   +     \sum_{k=k_0}^{k_0+N_0-M_0-1} \widehat{\varphi}\Big(\frac{k+\frac12}{a}\Big)
                                    - \Big(\frac{k+\frac12}{a}\Big) \\
 &   +     \sum_{m=m_0}^{m_0+M_0-1} \sigma\circ\widehat{\psi}\Big(\frac{m+\frac12}{b}\Big)
                                - \widehat{\psi}\Big(\frac{m+\frac12}{b}\Big)
   +     \sum_{m=m_0}^{m_0+M_0-1}\widehat{\psi}\Big(\frac{m+\frac12}{b}\Big)
                                - \Big(\frac{m+\frac12}{b}\Big) \\
 &    +     \sum_{m=m_0}^{m_0+M_0-1} \Big(\frac{m+\frac12}{b}\Big) -
         \sum_{k=k_0+N_0=M_0}^{k_0+N_0-1} \Big(\frac{k+\frac12}{a}\Big) \\
 \le & \  2\widehat{M}(N_0-M_0)+2\widehat{M}(M_0) + \frac{(m_0+M_0)(m_0+M_0-1)}{2b} - \frac{m_0(m_0-1)}{2b} + \frac{M_0}{2b} \\
    &   \hskip.25in
             -     \frac{(k_0+N_0)(k_0+N_0-1)}{2a}
             + \frac{(k_0+N_0-M_0) (k_0+N_0-M_0-1)}{2a} - \frac{M_0}{2a}\\
= &\    N_0\Big(2\widehat{M}-\frac{M_0}{a}\Big) + M_0\Big(\frac{m_0}{b} - \frac{k_0}{a}\Big) + \frac{M_0^2}{2}\Big(\frac{1}{a}+\frac{1}{b}\Big) \\
< &\    N_0\Big(2\widehat{M}-\frac{M_0}{a}\Big)
        + \Big(M_0+\frac{M_0^2}{2}\Big)\Big(\frac{1}{a}+\frac{1}{b}\Big)
\end{align*}
where we have used the observation that
$\big|\frac{m_0}{b} - \frac{k_0}{a}\big|\le \frac{1}{a}+\frac{1}{b}$.
Now choose $M_0$ so that $2\widehat{M}a<M_0$, then choose $K_0$ so that the corresponding $N_0$ satisfies
$N_0>M_0$ and $S(\varphi)<0$.

Similarly we can show that there is an injection $\varphi$ so that $S(\varphi)>0$ by choosing $0\le M_0\le N_0$ and letting
\begin{eqnarray*}
\varphi\Big(\frac{\Z+\frac12}{a}\,\cap\, I\Big)
&  =  & \widehat{\varphi}\Big(\Big\{\frac{k_0+M_0+\frac12}{a},\,\dots,\,
        \frac{k_0+N_0-1+\frac12}{a}\Big\}\Big) \\
&     & \cup\,\,\,\,
\widehat{\psi}\Big(\Big\{\frac{m_0+K_0-N_0-M_0+\frac12}{b},\,\dots,\,\frac{m_0+K_0-N_0-1+\frac12}{b}\Big\}\Big)
\end{eqnarray*}
and proceeding as before.

Next we will show that the largest distance between successive values of $S(\varphi)$ is
at most $\displaystyle{\frac{1}{a+b}+2\widehat{M}}$.  To that end, let
$\varphi$ be an injection of the form (\ref{eqn:injectivemap})
and observe that since
$\big(\frac{\Z+\frac12}{b}\big) \cap I \ne \emptyset$,
$$\Big(\frac{\Z+\frac12}{a+b} \cap I\Big)
  \setminus \varphi\Big(\frac{\Z+\frac12}{a} \cap I\Big) \ne\emptyset.$$
Now let
$x\in \big(\frac{\Z+{\frac 1 2}}{a+b}\cap I\big) \setminus \varphi\big(\frac{\Z+{\frac 1 2}}{a} \cap I\big)$
be adjacent to an element $y\in\varphi\big(\frac{\Z+{\frac 1 2}}{a} \cap I\big)$
and define the injection $\varphi'$ by
$$\varphi'\Big(\frac{k+\frac12}{a}\Big) =
\left\{\begin{array}{cl}
\varphi\big(\frac{k+\frac12}{a}\big)
			 & \mbox{if $\varphi\big(\frac{k+\frac12}{a}\big)\ne y$} \\
			 &    \\
     x       & \mbox{if $\varphi\big(\frac{k+\frac12}{a}\big) =  y$.}
       \end{array}\right.$$
Then
\begin{align*}
 &\left|\sum_{\frac{k+\frac12}{a}\in I} 
		 \sigma\circ\varphi\Big(\frac{k+\frac12}{a}\Big)
		 - \sum_{\frac{k+\frac12}{a}\in I}
		 \sigma\circ\varphi'\Big(\frac{k+\frac12}{a}\Big)
		 \right| 
 \le  \left|\sum_{\varphi(\frac{k+\frac12}{a})\ne y}
		\sigma\circ\varphi\Big(\frac{k+\frac12}{a}\Big)
		 - \sigma\circ\varphi'\Big(\frac{k+\frac12}{a}\Big)
		 \right| + |\sigma(y)-\sigma(x)| \\
  & \hspace{7cm} =    |\sigma(y)-\sigma(x)| 
 \le   |\sigma(y)-y| + |y-x| + |\sigma(x)-x| 
 \le   2\widehat{M} + \frac{1}{a+b}.
\end{align*}

It remains to show that any two such injections can be obtained from each other by a series of such operations.  This is obvious if we start with an injection $\varphi_0$ satisfying
$$\varphi_0\Big(\frac{\Z+\frac12}{a} \cap I\Big) =
\Big\{\frac{\ell_0+\frac12}{a+b},\,\dots,\,\frac{\ell_0+N_0-1}{a+b}\Big\}$$
and noting that the range of any other injection $\varphi$ can be
obtained by moving $\frac{\ell_0+N_0-1}{a+b}$ incrementally to the largest element of
$\varphi\big(\frac{\Z+{\frac 1 2}}{a} \cap I\big)$, then moving
$\frac{\ell_0+N_0-2}{a+b}$ to the second largest element of
$\varphi\big(\frac{\Z+{\frac 1 2}}{a} \cap I\big)$ and so on.

The above arguments imply that there is an injection $\varphi$ with the
property that
$$\left|\sum_{\frac{k+\frac12}{a}\in I}\sigma\circ\varphi\Big(\frac{k+\frac12}{a}\Big)
  - \frac{k+\frac12}{a}\right| \le \widehat{M}+\frac{1}{2(a+b)}.$$

We will now construct $\varphi$ and $\psi$ as required by the theorem.
Choose $K_0$ so large that $\widehat{M}(a+b)+\frac{1}{2}\le \delta\,K_0$, and
that (\ref{eqn:avdoninmap}) is satisfied for the $\delta$-Avdonin maps $\varphi$
and $\psi$, and for the $\epsilon$-Avdonin map $\sigma$.
Construct an injection $\varphi$ as above
on every interval of the form
$I=[\frac{m K_0}{a+b},\frac{(m+1)K_0}{a+b})$, $m\in\Z$.
Then construct an injection
$$\psi\colon \Big(\frac{\Z+\frac 1 2 }{b}\Big)\cap I\to
\Big(\frac{\Z+\frac12}{a+b}\,\Z \cap I\Big)
\setminus \varphi\Big(\frac{\Z+\frac12}{a} \cap I\Big)$$
by arbitrarily assigning values.  Clearly this implies the existence of injections
$$\varphi\colon \frac{\Z+\frac12}{a} \to \frac{\Z+\frac12}{a+b},\qquad
  \psi\colon \frac{\Z+\frac12}{b} \to \frac{\Z+\frac12}{a+b}$$
as required by the theorem.

By construction, it is clear that (a) is satisfied,
and note that for $j\in\Z$,
\begin{align*}
\Big|\sigma\circ\varphi\Big(\frac{j+\frac12}{a}\Big) - \frac{j+\frac12}{a}\Big|
 \le    \Big|\sigma\circ\varphi\Big(\frac{j+\frac12}{a}\Big)
          - \varphi\Big(\frac{j+\frac12}{a}\Big)\Big|
    +  \Big|\varphi\Big(\frac{j+\frac12}{a}\Big) - \frac{j+\frac12}{a}\Big| 
 \le   \widehat{M} + \frac{K_0}{a+b} = M.
\end{align*}
This gives (b). Finally note that for all $m\in\Z$,
\begin{align*}
\Big|\sum_{\frac{k+{\frac 1 2}}{a}
	    \in \big[\frac{mK_0}{a+b},\frac{(m+1)K_0}{a+b}\big)}
	    \sigma\circ\varphi\Big(\frac{k+\frac12}{a}\Big)
  - \frac{k+\frac12}{a}\Big|
 \le    \widehat{M}+\frac{1}{2(a+b)} 
 \le    \delta\,\frac{K_0}{a+b} 
 \le    (\epsilon+3\delta)\,\frac{K_0}{a+b}.
\end{align*}
and that
\begin{align*}
&      \Big|\sum_{\frac{k+{\frac 1 2}}{b} \in
		[\frac{mK_0}{a+b},\frac{(m+1)K_0}{a+b})}
		\sigma\circ\psi\Big(\frac{k+\frac12}{b}\Big)
  	    - \frac{k+\frac12}{b}\Big| \\
& \le  \Big|\sum_{\frac{k+{\frac 1 2}}{a} \in
		[\frac{mK_0}{a+b},\frac{(m+1K_0}{a+b})}
		\sigma\circ\varphi\Big(\frac{k+\frac12}{a}\Big) - \frac{k+\frac12}{a} 
    %\\
%&       \hskip.25in
            + \sum_{\frac{k+{\frac 1 2}}{b} \in
		[\frac{mK_0}{a+b},\frac{(m+1)K_0}{a+b})}
		\sigma\circ\psi\Big(\frac{k+\frac12}{b}\Big)
            - \frac{k+\frac12}{b}\Big| \\
&       \hskip.25in
            + \Big|\sum_{\frac{k+{\frac 1 2}}{a} \in
		[\frac{mK_0}{a+b},\frac{(m+1)K_0}{a+b})}
		\sigma\circ\varphi\Big(\frac{k+\frac12}{a}\Big)
            - \frac{k+\frac12}{a}\Big| \\
& \le   \Big|\sum_{\frac{k+{\frac 1 2}}{a} \in
		[\frac{mK_0}{a+b},\frac{(m+1)K_0}{a+b})}
        \sigma\circ\varphi\Big(\frac{k+\frac12}{a}\Big)
        - \widehat{\varphi}\Big(\frac{k+\frac12}{a}\Big) 
        %\\
%&       \hskip.25in
            + \sum_{\frac{k+{\frac 1 2}}{b} \in
		[\frac{mK_0}{a+b},\frac{(m+1)K_0}{a+b})}
		\sigma\circ\psi\Big(\frac{k+\frac12}{b}\Big)
            - \widehat{\psi}\Big(\frac{k+\frac12}{b}\Big)\Big| \\
&       \hskip.25in + \Big|\sum_{\frac{k+{\frac 1 2}}{a} \in
		[\frac{mK_0}{a+b},\frac{(m+1)K_0}{a+b})}
        \widehat{\varphi}\Big(\frac{k+\frac12}{a}\Big)
        - \frac{k+\frac12}{a}\Big| 
        %\\
%&       
\hskip.25in
        + \Big|\sum_{\frac{k+{\frac 1 2}}{b} \in
		[\frac{mK_0}{a+b},\frac{(m+1)K_0}{a+b})}
		\widehat{\psi}\Big(\frac{k+\frac12}{b}\Big)
            - \frac{k+\frac12}{b}\Big| \\
&       \hskip.25in +\Big|\sum_{\frac{k+{\frac 1 2}}{a} \in
		[\frac{mK_0}{a+b},\frac{(m+1)K_0}{a+b})}
		\sigma\circ\varphi\Big(\frac{k+\frac12}{a}\Big)
            - \frac{k+\frac12}{a}\Big| \\
& \le   \Big|\sum_{\frac{k+{\frac 1 2}}{a+b} \in
		[\frac{mK_0}{a+b},\frac{(m+1)K_0}{a+b})}
		\sigma\Big(\frac{k+\frac12}{a+b}\Big)
		- \frac{k+\frac12}{a+b}\Big|
           + 3\delta \frac{K_0}{a+b} < (\epsilon + 3\delta)\frac{K_0}{a+b}.
\end{align*}
This proves (c).
\end{proof}

\begin{figure}
  \begin{center}

  \begin{tikzpicture}[scale=1]

\begin{scope}
  \draw [-] (0,10) -- (12,10);
\!\!\!
  \draw [-] (0,9.8) -- (0,10.2);
  \node at (0,10.4) {$0$};

  \node at (1,9.4) {$\frac   {\Z+\frac 1 2}{b_1}$};

  \draw [-] (2.8,9.8) -- (2.8,10.2);
  \node at (1.4,10.4) {$b_1$};
  \node at (2.8,10.6) {$a_1$};

  \node at (3.7,9.4) {$\frac {\Z+\frac 1 2} {b_2} $};

  \draw [-] (4.6,9.8) -- (4.6,10.2);
  \node at (3.7,10.4) {$b_2$};
  \node at (4.6,10.6) {$a_2$};

  \node at (6.5,9.4) {$\frac {\Z+\frac 1 2} {b_3}$};

  \draw [-] (8.4,9.8) -- (8.4,10.2);
  \node at (6.5,10.4) {$b_3$};
  \node at (8.6,10.6) {$a_3$};

  \node at (9.4,9.4) {$\frac {\Z+\frac 1 2} {b_4}$};

  \draw [-] (10.4,9.8) -- (10.4,10.2);
  \node at (9.4,10.4) {$b_4$};
  \node at (10.35,10.6) {$a_4$};

  \node at (11.3,9.4) {$\frac {\Z+\frac 1 2} {b_5} $};

  \draw [-] (12,9.8) -- (12,10.2);
  \node at (11.1,10.4) {$b_5$};
  %\node at (10.6,10.6) {$a_5$};
  \node at (12,10.4) {$1$};
\end{scope}

\begin{scope}[shift={(10.3,7)}]
  \draw [right hook-] (-.7,2) -- (0,0.6);
  \draw [left hook-] (.8,2) -- (0,0.6);
  \draw [->>](0,.6) -- (0,.2);
  \node at (-.6,1.3) {$\widehat\varphi_4$};
  \node at (-.15,1.5) {$ \varphi_4$};
  \node at (0.3,1.65) {$\psi_4$};
  \node at (.7,1.25) {$\widehat\psi_4$};
\end{scope}

% level 3

\begin{scope}[shift={(0,-2.5)}]
  \draw [-] (0,10) -- (12,10);

  \draw [-] (0,9.8) -- (0,10.2);
  \node at (0,10.4) {$0$};

  \node at (10.3,9.4) {$\frac {\Z+\frac 1 2} {c_3}$ };

  \draw [-] (4.6,9.8) -- (4.6,10.2);
  %\node at (4.6,10.4) {$a_2$};

  \node at (6.5,9.4) {$\frac {\Z+\frac 1 2} {b_3} $};

  \draw [-] (8.4,9.8) -- (8.4,10.2);
  \draw [-] (12,9.8) -- (12,10.2);
  \node at (12,10.4) {$1$};
\end{scope}

\begin{scope}[shift={(8.28,4.5)}]
  \draw [right hook-] (-1.8,2) -- (0,0.6);
  \draw [left hook-] (1.7,2.1) -- (0,0.6);
  \draw [->>](0,.6) -- (0,.2);
  \node at (-1.3,1.3) {$\widehat\varphi_3$};
  \node at (-.7,1.5) {$ \varphi_3$};
  \node at (0.7,1.65) {$\psi_3$};
  \node at (1.2,1.25) {$\widehat\psi_3$};
\end{scope}

% level 2
\begin{scope}[shift={(0,-5)}]
  \draw [-] (0,10) -- (12,10);
  \draw [-] (0,9.8) -- (0,10.2);
  \node at (0,10.4) {$0$};
  \node at (3.9,9.4) {$\frac {\Z+\frac 1 2} {b_2}$};

   \node at (8.3,9.4) {$\frac {\Z+\frac 1 2} {c_2} $};
  \draw [-] (4.6,9.8) -- (4.6,10.2);
\draw [-] (2.8,9.8) -- (2.8,10.2);
  \draw [-] (12,9.8) -- (12,10.2);
  \node at (12,10.4) {$1$};
\end{scope}

\begin{scope}[shift={(7.2,2)}]
  \draw [right hook-] (-3,2) -- (0,0.6);
  \draw [left hook-] (1,2) -- (0,0.6);
  \draw [->>](0,.6) -- (0,.2);
  \node at (-2.2,1.4) {$\widehat\varphi_2$};
  \node at (-1.2,1.4) {$\varphi_2$};
  \node at (.2,1.4) {$\psi_2$};
  \node at (.9,1.4) {$\widehat\psi_2$};
\end{scope}

 % level 1

\begin{scope}[shift={(0,-7.5)}]
  \draw [-] (0,10) -- (12,10);

  \draw [-] (0,9.8) -- (0,10.2);
  \node at (0,10.4) {$0$};

\draw [-] (2.8,9.8) -- (2.8,10.2);

  \node at (7.2,9.4) {$\frac {\Z+\frac 1 2} {c_1} $};

  \node at (1.4,9.4) {$\frac {\Z+\frac 1 2} {b_1} $};

  \draw [-] (12,9.8) -- (12,10.2);
  \node at (12,10.4) {$1$};
\end{scope}

\begin{scope}[shift={(6,-.5)}]
  \draw [left hook-] (1,2) -- (0,0.6);
  \draw [right hook-] (-4,2.1) -- (0,0.6);
  \draw [->>](0,.6) -- (0,.2);
%  \node at (-1.4,1.4) {$\varphi_1=\widehat{\varphi}_1$};
%  \node at (1,1.4) {$\psi_1=\widehat{\psi}_1$};
  \node at (-0.9,1.4) {$\varphi_1=\widehat{\varphi}_1$};
  \node at (1.4,1.4) {$\psi_1=\widehat{\psi}_1$};
\end{scope}

%target

\begin{scope}[shift={(0,-10)}]
  \draw [-] (0,10) -- (12,10);

  \draw [-] (0,9.8) -- (0,10.2);
  \node at (0,10.4) {$0$};

  \node at (6,9.4) {$ \Z+\frac 1 2$};

  \draw [-] (12,9.8) -- (12,10.2);
  \node at (12,10.4) {$1$};
\end{scope}

 %\draw [right hook->] (2,6.5) -- (5.6,-.4);
 % \node  at (2.1,4) {$\sigma_2{\circ}\varphi_2=\sigma_1$};

 %\draw [right hook->] (4.0,4) -- (5.8,-.3);
  %\node at (3.6,2.1) {$\widehat{\varphi}_4{\circ}\varphi_3=\sigma_2$};

  \end{tikzpicture}

  \end{center}\label{fig:successiveapplication}
  \caption{
  Proof by successive application of Theorem~\ref{thm:rounding} and Theorem~\ref{thm:rearrangement}. Theorem~\ref{thm:rounding} provides us with the Avdonin maps $\widehat{\varphi}_k$, $\widehat{\psi}_k$. Using Theorem~\ref{thm:rearrangement}, we can first adjust $\widehat{\varphi}_2$ and $\widehat{\psi}_2$ to obtain Avdonin maps $\varphi_2$ and $\psi_2$ with the property that $\psi_1\circ \varphi_2$ and $\psi_1\circ \psi_2$ are Avdonin maps.  Next, we update $\widehat{\varphi}_3$ and $\widehat{\psi}_3$ to  
$\varphi_3$ and $\psi_3$ with the property that $\psi_1\circ \psi_2 \circ \varphi_3$ and $\psi_1\circ \psi_2 \circ \phi_3$  are Avdonin maps, and so on.
  }
\end{figure}
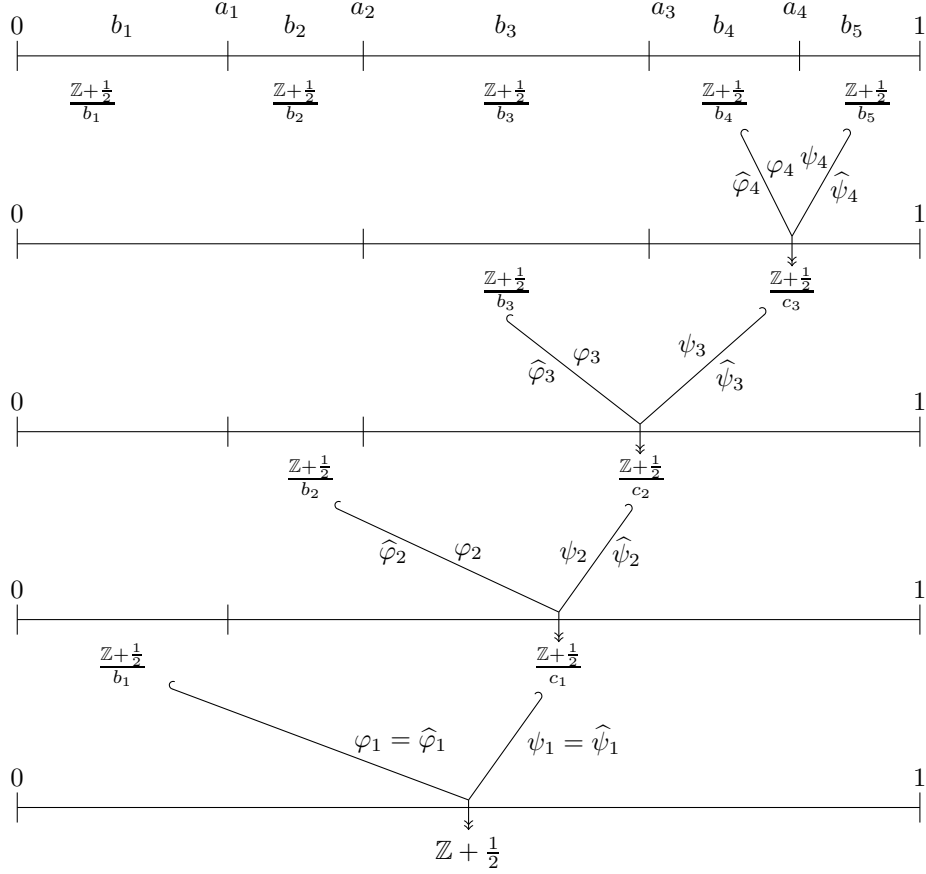

\section{Adjoining intervals}\label{sec:comb}
The tools derived so far allow us to prove Theorem~\ref{thm:main1} in the finite and the countable setting.  In order to prove Theorem~\ref{thm:main1}, we need to show how Riesz bases of intervals of length $a$ and $b$ can be combined to a Riesz bases of intervals of length $a+b$.

\begin{theorem}\label{lem:implicitmap}
Let $a,\,b>0$ and suppose that
$$\tau\colon\frac{\Z+\frac12}{a} \to \Z+\frac12,\qquad
\mbox{and}\qquad
  \eta\colon\frac{\Z+\frac12}{b} \to \Z+\frac12$$
are $\epsilon$-Avdonin maps with disjoint ranges.  Then there exists a $4\epsilon$-Avdonin map
$$\rho\colon \frac{\Z+\frac12}{a+b} \to \Z+\frac12$$
such that
$$\rho\Big(\frac{\Z+\frac12}{a+b}\Big) = \tau\Big(\frac{\Z+\frac12}{a}\Big)
    \,\cup\,\eta\Big(\frac{\Z+\frac12}{b}\Big).$$
\end{theorem}

\begin{proof}
In order to define $\rho$ on $\frac{\Z+\frac12}{a+b}$, fix an interval of the form
$I=[\frac{\ell K_0}{a+b},\frac{(\ell+1) K_0}{a+b})$ where $\ell\in\Z$ and $K_0$ is an integer to be chosen
later.  We will define $\rho$ on $\frac{\Z+\frac12}{a+b}\,\cap\,I$ for each $\ell$.

 Let $\varphi$ and $\psi$ be rounding maps given in Theorem~\ref{thm:rounding}.
 Then for any $\epsilon>0$, $\varphi$ and $\psi$ are $\epsilon$-Avdonin maps and there exist natural numbers $K_0$ arbitrarily large such that 
 $$\varphi\colon \frac{\Z+\frac12}{a} \cap I \to \frac{\Z+\frac12}{a+b}\,\cap\,I,\ \mbox{and}\ 
   \psi\colon \frac{\Z+\frac12}{b} \cap I \to \frac{\Z+\frac12}{a+b}\,\cap\,I$$
are injective maps with disjoint ranges.  Also recall that by Lemma~\ref{lem:localpartition} the ranges of $\varphi$ and $\psi$ restricted to $I$ partition 
$\frac{\Z+\frac12}{a+b}\cap I$.  Now define the map 
$$\rho=(\tau\circ \varphi^{-1}) \cup (\eta\circ\psi^{-1})\colon \frac{\Z+\frac12}{a+b} \to \Z+\frac12$$
by 
\begin{align*}
  \rho\Big(\as{m}{a+b}\Big)=
  \begin{cases}
    \tau \circ \varphi^{-1}\big(\frac{m+{\frac 1 2}}{a+b}\big), \quad \text{if}\quad \frac{m+{\frac 1 2}}{a+b} \quad \text{is in the range of }\varphi;\\
    \eta \circ \psi^{-1}\big(\frac{m+{\frac 1 2}}{a+b}\big), \quad \text{if}\quad \frac{m+{\frac 1 2}}{a+b} \quad \text{is in the range of }\psi.
  \end{cases}
\end{align*}
We compute
\begin{eqnarray}
&     &  \bigg|\sum_{\frac{m+{\frac 1 2}}{a+b}\in I}
      \rho\Big(\frac{m+\frac12}{a+b}\Big) - \Big(\frac{m+\frac12}{a+b}\Big)\bigg|\notag \\
& \le & \bigg|\sum_{\frac{m+{\frac 1 2}}{a+b}\in I}
      (\tau\circ \varphi^{-1}) \cup (\eta\circ\psi^{-1})\Big(\frac{m+\frac12}{a+b}\Big) -
      \sum_{\frac{k+{\frac 1 2}}{a}\in I} \tau\Big(\frac{k+\frac12}{a}\Big) -
      \sum_{\frac{j+{\frac 1 2}}{b}\in I} \eta\Big(\frac{j+\frac12}{b}\Big)\bigg|
      \label{eqn:implicitmap1}  \\
&  +  & \bigg|\sum_{\frac{k+{\frac 1 2}}{a}\in I} \tau\Big(\frac{k+\frac12}{a}\Big) -
        \Big(\frac{k+\frac12}{a}\Big)\bigg| +
      \bigg|\sum_{\frac{j+{\frac 1 2}}{b}\in I} \eta\Big(\frac{j+\frac12}{b}\Big)  -
        \Big(\frac{j+\frac12}{b}\Big)\bigg|
      \label{eqn:implicitmap2}  \\
&  +  & \bigg|\sum_{\frac{k+{\frac 1 2}}{a}\in I} \Big(\frac{k+\frac12}{a}\Big) -
        \varphi\Big(\frac{k+\frac12}{a}\Big)\bigg| +
      \bigg|\sum_{\frac{j+{\frac 1 2}}{b}\in I} \Big(\frac{j+\frac12}{b}\Big)  -
        \psi\Big(\frac{j+\frac12}{a}\Big)\bigg|
      \label{eqn:implicitmap3}   \\
&  +  &  \bigg|\sum_{\frac{k+{\frac 1 2}}{a}\in I} \varphi\Big(\frac{k+\frac12}{a}\Big) +
      \sum_{\frac{j+{\frac 1 2}}{b}\in I} \psi\Big(\frac{j+\frac12}{b}\Big)  -
      \sum_{\frac{m+{\frac 1 2}}{a+b}\in I}  \Big(\frac{m+\frac12}{a+b}\Big)\bigg|.
      \label{eqn:implicitmap4}
\end{eqnarray}
By definition of $\rho$, no matter which $K_0$ is chosen, (\ref{eqn:implicitmap1}) $=0$
and again by \eqref{eqn:localpartition}, (\ref{eqn:implicitmap4}) also vanishes.
Since $\tau$ and $\eta$ are $\epsilon$-Avdonin maps, there exists $R_0$ such that if
$\frac{K_0}{a+b}>R_0$ then (\ref{eqn:implicitmap2}) is less than $2\epsilon$.  Finally note
that by Theorem~\ref{thm:rounding}(b), there is an $R_1>0$ such that if $\frac{K_0}{a+b}>R_1$
then (\ref{eqn:implicitmap3}) is less than $2\epsilon$.
\end{proof}

\begin{figure}
\begin{center}
 \begin{tikzpicture}[scale=.75]

\begin{scope}
  \draw [-] (0,10) -- (12,10);

  \draw [-] (0,9.8) -- (0,10.2);
  \node at (0,10.4) {$0$};

  %\node at (1.4,9.4) {$\frac \Z {a_1}$};

  \draw [-] (2.8,9.8) -- (2.8,10.2);
  \node at (2.8,10.4) {$a_1$};

  %\node at (3.7,9.4) {$\frac {\Z+\frac 1 2 } {a_2-a_1}$};

  \draw [-] (4.6,9.8) -- (4.6,10.2);
  \node at (4.6,10.4) {$a_2$};

  \node at (6.2,9.4) {$\frac {\Z+\frac 1 2 } {a_3-a_2} $};

  \draw [-] (8.4,9.8) -- (8.4,10.2);
  \node at (8.4,10.4) {$a_3$};

  \node at (9.4,9.4) {$\frac {\Z+\frac 1 2 } {a_4-a_3} $};

  \draw [-] (10.4,9.8) -- (10.4,10.2);
  \node at (10.4,10.4) {$a_4$};

 % \node at (11.2,9.4) {$\frac {\Z+\frac 1 2 } {1-a_4} $};

  \draw [-] (12,9.8) -- (12,10.2);
  \node at (12,10.4) {$1$};
\end{scope}

  \begin{scope}[shift={(7.4,7)}]
  \draw [right hook-] (-1.3,1.8) -- (0,0.6);
  \draw [left hook-] (1.9,1.8) -- (0,0.6);
  \draw [->>](0,.6) -- (0,.2);
  %\node at (-.8,1.2) {$\widehat\varphi_3$};
  \node  at (-.45,1.5) {$\phi$};
  \node  at (.8,1.55) {$\psi$};
  %\node at (1.9,1.2) {$\widehat\psi_3$};
\end{scope}

\begin{scope}[shift={(0,-2.5)}]
  \draw [-] (0,10) -- (12,10);

  \draw [-] (0,9.8) -- (0,10.2);
  \node at (0,10.4) {$0$};

  %\node at (2.3,9.4) {$\frac 1 {a_2} \Z$};

  \draw [-] (4.6,9.8) -- (4.6,10.2);
  \node at (4.6,10.4) {$a_2$};

  \node  at (7.4,9.25) {$\frac{\Z+\frac 1 2 }  {a_4-a_2}$};

  %\draw [-] (8.4,9.8) -- (8.4,10.2);
  %\node at (8.4,10.4) {$a_3$};

  %\node at (9.4,9.4) {$\frac 1 {a_4-a_3} \Z$};

  \draw [-] (10.4,9.8) -- (10.4,10.2);
  \node at (10.4,10.4) {$a_4$};

  %\node at (11.2,9.4) {$\frac 1 {1-a_4} \Z$};

  \draw [-] (12,9.8) -- (12,10.2);
  \node at (12,10.4) {$1$};
\end{scope}

\begin{scope}[shift={(0,-5)}]
  \draw [-] (0,10) -- (12,10);

  \draw [-] (0,9.8) -- (0,10.2);
  \node at (0,10.4) {$0$};

 % \node at (4.2,9.4) {$\frac 1 {a_3} \Z$};

 % \draw [-] (8.4,9.8) -- (8.4,10.2);
 % \node at (8.4,10.4) {$a_3$};

  % \node at (9.4,9.4) {$\frac 1 {a_4-a_3} \Z$};

  %\draw [-] (10.4,9.8) -- (10.4,10.2);
 % \node at (10.4,10.4) {$a_4$};

  %\node at (11.2,9.4) {$\frac 1 {1-a_4} \Z$};

  \draw [-] (12,9.8) -- (12,10.2);
  \node at (12,10.4) {$1$};
\end{scope}

\begin{scope}[shift={(0,-7.5)}]
  \draw [-] (0,10) -- (12,10);

  \draw [-] (0,9.8) -- (0,10.2);
  \node at (0,10.4) {$0$};

  %\node at (5.6,9.4) {$\frac 1 {a_4} \Z$};

  %\draw [-] (10.4,9.8) -- (10.4,10.2);
 % \node at (10.4,10.4) {$a_4$};

 % \node at (11.2,9.4) {$\frac 1 {1-a_4} \Z$};

  \draw [-] (12,9.8) -- (12,10.2);
  \node at (12,10.4) {$1$};
\end{scope}

\begin{scope}[shift={(0,-10)}]
  \draw [-] (0,10) -- (12,10);

  \draw [-] (0,9.8) -- (0,10.2);
  \node at (0,10.4) {$0$};

  \node at (6,9.4) {$\Z$};

  \draw [-] (12,9.8) -- (12,10.2);
  \node at (12,10.4) {$1$};
\end{scope}
\draw [-] (5,2) -- (6,.3);
\draw [-] (5.5,2.2) -- (6,.3);
\draw [-] (7,2) -- (6,.3);
\draw [right hook-] (5.8,8.7) -- (6,.3);
\draw [left hook-] (9.6,8.7) -- (6,.3);
\draw [left hook-] (7.2,6.2) -- (6,.3);
 \draw [->>](6,.3) -- (6,-.3);
 \node at (8,4) {$\eta$};
 \node at (7,4) {$\rho$};
 \node at (5.5,4) {$\tau$};

  \end{tikzpicture}
  \end{center}
\caption{The map $\rho$ is implicitly defined by the maps $\tau$ and $\eta$. We shall show that $\rho$ is an Avdonin map. }
\end{figure}
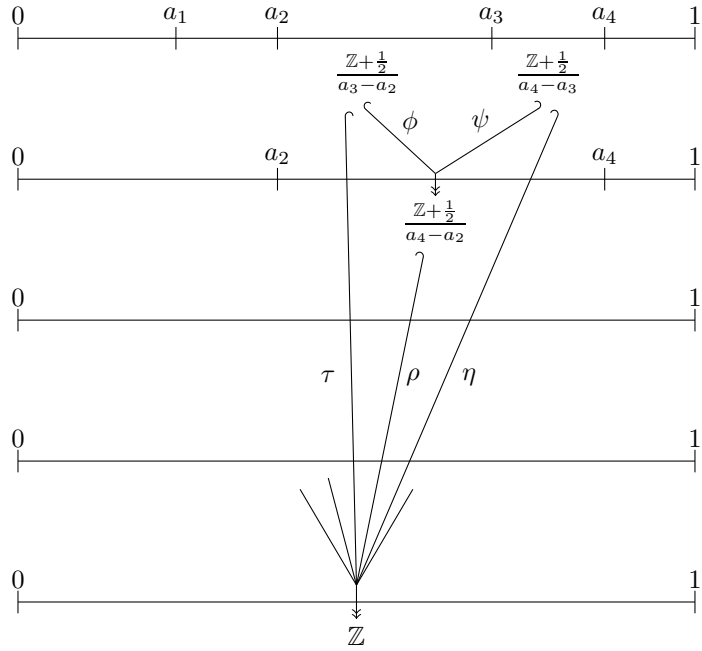

\section{Proof of Theorems~\ref{thm:main2} and~\ref{thm:maincountable}}\label{sec:proof} 

It remains to employ our technical results on Avdonin sequence, namely Theorems~\ref{thm:rounding},~\ref{thm:rearrangement} and \ref{lem:implicitmap} to prove the main results of this paper.

\smallskip
\noindent {\it Proof of Theorem~\ref{thm:main2} and Theorem~\ref{thm:maincountable}}.
\smallskip

The following construction is illustrated in  Figure~\ref{fig:successiveapplication}.
Set $c_j=\sum_{k=j+1}^n b_j$ for $j=1,\ldots,n-1$ in the finite case described in Theorem~\ref{thm:main2} and $c_j=\sum_{k=j+1}^\infty b_j$, $j\in\N$ in the countable setting of Theorem~\ref{thm:maincountable}.  As maps, we choose 
\begin{align*}
\widehat{\varphi}_j: \frac {\Z+\frac 1 2}{b_j} \to \frac {\Z+\frac 1 2} {b_j+c_j},  \quad \widehat{\psi}_j: \frac {\Z+\frac 1 2}{c_j} \to \frac {\Z+\frac 1 2} {b_j+c_j}
\end{align*}
as defined in Theorem~\ref{thm:rounding}. These  are $\epsilon$-Avdonin maps for all $\epsilon>0$, and each pair partitions the respective set 
$\frac {\Z+\frac 1 2} {b_j+c_j} = \frac {\Z+\frac 1 2} {c_{j-1}}$.   We can take $\widehat{\varphi}_j$ and $\widehat{\psi}_j$ to be the simple rounding maps defined 
in Theorem~\ref{thm:rounding} and Lemma~\ref{lem:localpartition}.  As such, each mapping satisfies the additional localization property that allows for the
application of Theorem~\ref{thm:rearrangement}.

For fixed $K$
set $\delta = 4^{-K}$,  $\epsilon_1=\frac \delta 2$ and $\epsilon_j=\frac \delta {2^j\cdot 3}$ for $j\geq 2$. 
By construction, $\Phi_1=\varphi_1=\widehat{\varphi}_1$ and   $\Psi_1=\psi_1=\widehat{\psi}_1$ are $\epsilon_1=\frac \delta 2$-Avdonin maps.  
Now, we apply Theorem~\ref{thm:rearrangement} to the 
$\epsilon_2=\frac \delta {12}$-Avdonin maps $\widehat{\varphi}_2$ 
and $\widehat{\psi}_2$ 
with $\Psi_1=\widehat{\psi}_1$ playing the role of $\sigma$.
Thus we obtain maps ${\varphi}_2$ and ${\psi}_2$ with the property that  $\Phi_2=\psi_1\circ \varphi_2$ and $\Psi_2=\psi_1\circ\psi_2$ are $\epsilon_1+3\epsilon_2=\frac \delta 2 +3\frac \delta {12}=(1-\frac 1 4)\delta$-Avdonin maps.

Similarly, we consider $\widehat{\varphi}_3$ and $\widehat{\psi}_3$ as $\epsilon_3=\frac \delta {2^3 3}$-Avdonin maps and apply Theorem~\ref{thm:rearrangement}
once again with now $\Psi_2$ playing the role of $\sigma$.  This yields maps ${\varphi}_3$ and ${\psi}_3$ so that 
$\Phi_3=\Psi_2\circ\varphi_3 = \psi_1\circ \psi_2\circ \varphi_3$ and $\Psi_3=\Psi_2\circ\psi_3 = \psi_1\circ \psi_2\circ \psi_3$ are 
$(1-\frac 1 4)\delta+ 3 \frac \delta {2^3 3}=(1-\frac 1 8)\delta$-Avdonin maps. 

In general, assuming that $(1-\frac 1{2^k})\delta$-Avdonin maps
\begin{align*}
\Phi_k:\al{b_k} \to \Z+\frac 1 2,\ \mbox{and}\ \Psi_k:\al{c_k} \to \Z+\frac 1 2
\end{align*}
have been constructed, we apply Theorem~\ref{thm:rearrangement} to the $\frac{\delta}{2^{k+1}3}$-Avdonin maps
$\widehat{\varphi}_{k+1}$ and $\widehat{\psi}_{k+1}$ with $\Psi_k$ playing the role of $\sigma$.
This yields maps ${\varphi}_{k+1}$ and ${\psi}_{k+1}$ so that 
$\Phi_{k+1}=\Psi_k\circ\varphi_{k+1}$ and $\Psi_{k+1}=\Psi_k\circ\psi_{k+1}$ are 
$(1-\frac 1 {2^k})\delta+ 3 \frac \delta {2^{k+1} 3}=(1-\frac 1 {2^{k+1}})\delta$-Avdonin maps. 

In this way we obtain a sequence of maps
\begin{align*}
\Phi_j:\al{b_j} \to \Z+\frac 1 2,\ \mbox{and}\ \Psi_j:\al{c_j} \to \Z+\frac 1 2,\ j\in\N,
\end{align*}
with the property that $\Phi_j$ and $\Psi_j$ are $\delta$-Avdonin maps and hence that the sets
$
  \Lambda_j=\Phi_j(\al{b_j})\subseteq\Z+\frac1 2,
$
are disjoint and $\mathcal E(\Lambda_j)$ is a Riesz basis for $L^2(I)$ for any interval $I$ of length $b_j$.

In the setting of Theorem~\ref{thm:main2}, that is, if $b_j$ is the finite sequence $b_1,\,b_2,\,\dots,\,b_n$, we set $\Phi_n=\Psi_{n-1}$
and obtain in addition that $\{\Lambda_j\}_{j=1}^n$ partitions $\Z+\frac 1 2$. 

In order to complete the proof of Theorem~\ref{thm:main2} and to prove Theorem~\ref{thm:maincountable}
pick for $k\leq K$ indices $J = \{j_1,\ldots,j_k\}\subseteq \{1,\ldots,n\}$ (resp. $\N$).  Apply Theorem~\ref{lem:implicitmap}
to the $\delta$-Avdonin maps $\Phi_{j_1}$ and $\Phi_{j_2}$ and obtain a $4\delta$-Avdonin map
\begin{align*}
   \Rho_{\{j_1,j_2\}}: \al{b_{j_1}+b_{j_2}}\longrightarrow \Lambda_{j_1}\cup\Lambda_{j_2}\subseteq \Z+\frac 1 2.
\end{align*}
Applying Theorem~\ref{lem:implicitmap} again to $\Rho_{\{j_1,j_2\}}$ and $\Phi_{j_3}$ we obtain an $8\delta$-Avdonin map
\begin{align*}
   \Rho_{\{j_1,j_2,j_3\}}: \al{b_{j_1}+b_{j_2}+b_{j_3}}\longrightarrow \Lambda_{j_1}\cup\Lambda_{j_2}\cup\Lambda_{j_3}\subseteq \Z+\frac 1 2.
\end{align*}
Repeating this argument, we ultimately obtain a $4^{K-1}\delta=\frac 1 4$-Avdonin map
\begin{align*}
  \Rho_J  : \al{\sum_{j\in J}b_j}\longrightarrow \bigcup_{j\in J}\Lambda_j \subseteq \Z+\frac 1 2
\end{align*}
so that $\E(\Lambda_{j_1}\cup \ldots \cup \Lambda_{j_k})$ is a Riesz bases for $L^2(I)$ for any interval $I$ of length $b_{j_1}+ \ldots + b_{j_k}$.
This completes the proof of Theorems~\ref{thm:main2} and \ref{thm:maincountable}.

 \hfill $\Box$

\section{Consequences for sampling theory}\label{sec:sampling}

The unitary Fourier transform $\mathcal F$ is densily defined by 
\begin{align*}  
  \mathcal F:L^2(\R)\to L^2(\R), \quad f\mapsto \widehat f (y)
  =\mathcal F f(y)=\int_{\R} f(x)\, e^{-2\pi i xy } \, dx.
\end{align*}
It holds $\mathcal F^{-1}f(x)=\mathcal F f(-x)$.
The Paley Wiener space $PW[-\frac a 2 ,\frac a 2 ]$ consists of functions bandlimited to the interval $[-\frac a 2 ,\frac a 2 ]$, that is, of those $L^2(\R)$-functions whose Fourier transform is supported on $[-\frac a 2 ,\frac a 2 ]$. As $PW[ -\frac a 2 ,\frac a 2  ]=\mathcal F^{-1} L^2[-\frac a 2 ,\frac a 2 ]$, any orthonormal, orthogonal or Riesz bases $\mathcal E (\Lambda)$ on $L^2[-\frac a 2 ,\frac a 2 ]$ gives rise to an orthonormal, orthogonal or Riesz bases of $PW[-\frac a 2 ,\frac a 2 ]$ given by
\begin{align*}
  \mathcal F^{-1}\mathcal E (\Lambda)
  =\left\{\frac {\sin \pi a(\cdot +\lambda)}{\pi (\cdot + \lambda)}\right\}_{\lambda \in \Lambda}.
\end{align*}

Before stating our result in terms of sampling theory, we require the following definitions.

\begin{definition}
A set $\Gamma\subseteq \R$ is called a stable set of sampling for $PW[-\frac a 2, \frac a 2]$ if the map $f\mapsto \{ f(\gamma) \}_{\gamma \in \Gamma}$ is injective and boundedly invertible, and a set of interpolation for $PW[-\frac a 2, \frac a 2]$ if for each set $\{ c_\gamma \}_{\gamma \in \Gamma}$ there exists at least one $f\in PW[-\frac a 2, \frac a 2]$ with $f(\gamma)=c_\gamma$.
\end{definition}

Necessary criteria for a separated set $\Gamma$ to be a set sampling and of interpolation $\Gamma$ for $PW[-\frac a 2, \frac a 2]$ can be formulated using the  lower and upper Beurling densities of $\Gamma$. They  are given by
\begin{align*}
  D^{-}(\Gamma)=\lim_{r\to \infty} \inf_{a\in\R} \frac{\#\{\gamma\in \Gamma\cap [a,a+r)\}}r, \quad
  D^{+}(\Gamma)=\lim_{r\to \infty} \sup_{a\in\R} \frac{\#\{\gamma\in \Gamma\cap [a,a+r)\}}r.
\end{align*}
If $D^{-}(\Gamma)=D^{+}(\Gamma)$, then $\Gamma$ is said to have Beurling density $D(\Gamma)=D^{-}(\Gamma)=D^{+}(\Gamma)$.
\begin{theorem}
  Let $\Gamma$ be separated.
  \begin{itemize}
    \item[(a)] If $\Gamma$ is a stable set of sampling for $PW[-\frac a 2, \frac a 2]$ then $D^{-}(\Gamma)\geq a$.
    \item[(b)] If $\Gamma$ is a set of interpolation for $PW[-\frac a 2, \frac a 2]$ then $D^{+}(\Gamma)\leq a$.
  \end{itemize}
\end{theorem}
Clearly, the density does not characterize sets of sampling and of interpolation.  For example, $\Z\setminus\{0\}$ has Beurling density 1 but is not a set of sampling for $PW[-\frac 1 2, \frac 1 2]$ and $\Z\cup\{\frac 1 2\}$ has Beurling density 1 but is not a set of interpolation for $PW[-\frac 1 2, \frac 1 2]$.

\begin{theorem}
  Let $b_1,\ldots,b_n>0$ satisfy $\sum_{j=1}^n b_j=B$. Then exist pairwise disjoint sets $\Lambda_1,\Lambda_2,\ldots,\Lambda_n\subseteq \frac 1 B \Z$ with $\bigcup_{j=1}^n \Lambda_j=\frac 1 B \Z$, $D(\Lambda_j)=b_j$ and the property that
for any $J\subseteq\{1,\,2,\,\dots,\,n\}$, 
  $ \bigcup_{j\in J}\Lambda_j$ is a set of sampling for 
  $PW[-\frac a 2, \frac a 2]$ as long as 
  \begin{align*}
    D(\bigcup_{j\in J}\Lambda_j)=\sum_{j\in J} D(\Lambda_j)\geq a.
  \end{align*}  If $\sum_{j\in J} D(\Lambda_j)= a$, then $ \bigcup_{j\in J}\Lambda_j$ is also a  set of interpolation for $PW[-\frac a 2, \frac a 2]$.
\end{theorem}

We omit restating Theorem~\ref{thm:maincountable} in the sampling theory setting.

\section{Appendix}

\subsection{Proof of Theorem~\ref{thm:BF}}\label{appendix:BF}

  For $N\in\Z$ positive we let  $F(a,N)$ denote the number of elements of $\Big\lfloor{\almin{a}}\Big\rfloor=\Big\lfloor{\al{a}}\Big\rfloor$ that lie in $[0,N)$. The biggest $K$ with $\frac {K-\frac 1 2 } a < N$ is $\lfloor a N + \frac 1 2\rfloor$ and the smallest such $K$ is 1, consequently we have $F(a,N)=\lfloor a N + \frac 1 2\rfloor$. Irrationality of $a$ and $1-a$ implies
    \begin{eqnarray*}
     a N + \frac 1 2-1<&F(a,N)&<a N + \frac 1 2,\\
     (1-a) N + \frac 1 2-1<&F(1-a,N)&<
     (1-a) N +\frac 1 2.
     \end{eqnarray*}
     Addition of both equations gives
      \begin{align*}
    N-1<F(a,N)+F(1-a,N)<N+1
     \end{align*} and, therefore, $F(a,N)+F(1-a,N)=N$.

     Clearly, this implies that exactly one series assumes the value $0$ which is the only integer in $[0,1)$, and moving inductively from the interval $[0,N)$ to $[0,N+1)$ we add exactly one additional integer, namely, $N$, which is assumed by exactly one element  of one the  sequences $\Big\lfloor{\almin{a}}\Big\rfloor$ and $\Big\lfloor{\almin{1-a}}\Big\rfloor$.

Similarly, we note that $\frac {K-\frac 1 2 } a \in [-N,0)$ for $K=\lceil -a N +\frac 1 2 \rceil, \ldots,-1,0$, that is, for $G(a,N)=-\lceil -a N +\frac 1 2 \rceil + 1$ Elements. Irrationality gives, for example,
\begin{align*}
   -a N +\tfrac 1 2 < \lceil -a N +\tfrac 1 2 \rceil<  -a N +\tfrac 1 2 +1
\end{align*}
and
\begin{align*}
   a N -\frac 1 2 > - \lceil -a N +\frac 1 2 \rceil   >  - a N -\frac 1 2 -1.
\end{align*}
Using $G(a,N)=-\lceil -a N +\frac 1 2 \rceil + 1$, we obtain
\begin{eqnarray*}
  a N +\tfrac 1 2 >& G(a,N)&>     a N -\tfrac 1 2, \\
  (1-a) N +\tfrac 1 2 > &G(1-a,N)&>     (1-a) N -\tfrac 1 2
\end{eqnarray*}
which in sum gives
\begin{align*}
   N +1 &> G(a,N)+ G(1-a,N)>      N -1
\end{align*}
and $G(a,N)+ G(1-a,N)=N$ follows.  This implies inductively, as above, that each negative integer is assumed by exactly one element from one of the two sequences $\Big\lfloor{\almin{a}}\Big\rfloor$ and $\Big\lfloor{\almin{1-a}}\Big\rfloor$.

\subsection{Proof of Weyl-Khinchin Equidistribution}

To prove the following generalization of Theorem~\ref{thm:WK} requires no additional arguments.   
\begin{theorem}\label{thm:WKgeneral}
For a piecewise continuous function $f$ on $[0,1]$,   $a$ irrational and $\epsilon>0$ exists
$R_0$  so that for all $R\in\Z$, $R>R_0$, $\alpha \in \R$ and $m\in\Z$,
\begin{align*}
\Big|\frac{1}{R} \sum_{k=mR}^{(m+1)R-1}  f\Big(\frac {k+\alpha} a -
\Big\lfloor\frac {k+\alpha} a \Big\rfloor\Big)
 - \int_0^1 f(x)\, dx\Big| <\epsilon.
\end{align*}
\end{theorem}
\begin{proof}
Let us first consider the function $f(x) = e^{2\pi i\ell x}$ for some $\ell\in\Z$, $\ell\ne 0$. It holds
\begin{align*}
\lim_{R\to\infty}
 \frac{1}{R}\sum_{k=0}^{R-1} e^{2\pi i \ell \big(\frac {k+\alpha} a -
\big\lfloor\frac {k+\alpha} a \big\rfloor\big)}
  &=e^{2\pi i  \frac {\alpha \ell } a }
\lim_{R\to\infty}
 \frac{1}{R}\sum_{k=0}^{R-1} e^{2\pi i  \frac {k\ell} a }
=e^{2\pi i  \frac {\alpha \ell } a } \lim_{R\to\infty}
\frac{ 1-
	e^{2\pi i \frac{R k\ell}{a}}}{ R(1-e^{2\pi i \frac {k\ell } a})}
  =    0 = \int_0^1 e^{2\pi i \ell x}\, dx
  .
\end{align*}
Choosing $R_0$ so that if $R\ge R_0$,
$$\Big|\frac{1}{R}\sum_{n=0}^{R-1} e^{2\pi i  \frac  {k\ell}a }\Big| < \epsilon$$
we note that
\begin{align*}
 \frac{1}{R}  \sum_{k=mR}^{(m+1)R-1} e^{2\pi i  \frac  {k\ell}a}
 =e^{2\pi i  \frac  {mR\ell}a}\frac{1}{R}\sum_{k=0}^{ R-1} e^{2\pi i  \frac  {k\ell}a}
\end{align*}
 implies that
\begin{align*}
\Big|
 \frac{1}{R}\sum_{k=mR}^{(m+1)R-1} e^{2\pi i \ell \big(\frac {k+\alpha} a -
\big\lfloor\frac {k+\alpha} a \big\rfloor\big)} \Big|<\epsilon
\end{align*}
for all $m\in\N$ and all $R>R_0$, irrespectively the choice of $\alpha$, so the result is shown for all exponential functions. (The case $\ell=0$ holds trivially.) Using the triangle inequality the result extens to trigonometric polynomials.

The Stone-Weierstrass Theorem implies that for a an arbitrary continuous function $f$ on the torus, there exists for given $\epsilon>0$ a trigonometric polynomial $P$ with $|f(x)-P(x)|<\frac \epsilon 3$ for all $x\in [0,1]$. It follows
with $R_0$ chosen sufficiently large, that
$R>R_0$,
\begin{align*}
\Big|\frac{1}{R} &\sum_{k=mR}^{(m+1)R-1}  f\Big(\frac {k+\alpha} a -
\Big\lfloor\frac {k+\alpha} a \Big\rfloor\Big)
 - \int_0^1 f(x)\, dx\Big|
 \leq
 \Big|\frac{1}{R} \sum_{k=mR}^{(m+1)R-1}  \Big(f\Big(\frac {k+\alpha} a -
\Big\lfloor\frac {k+\alpha} a \Big\rfloor\Big)-P\Big(\frac {k+\alpha} a -
\Big\lfloor\frac {k+\alpha} a \Big\rfloor\Big)\Big)\Big|\\
&+
\Big|\frac{1}{R} \sum_{k=mR}^{(m+1)R-1}  P\Big(\frac {k+\alpha} a -
\Big\lfloor\frac {k+\alpha} a \Big\rfloor\Big)
 - \int_0^1 P(x)\, dx\Big|+
  \int_0^1 P(x)-f(x)\, dx\Big| <3 \frac \epsilon 3=\epsilon.
\end{align*}

If $f$ is piecewise continuous we can pick continuous functions $h$ and $g$ with
\begin{align*}
-M\leq g(x)\leq f(x)\leq h(x)\leq M
\end{align*}
and $g(x)=f(x)=h(x)$ outside a (possibly infinite) union of intervals $I$ of maximum length $\frac \epsilon M$. With $R_0$  sufficiently large, we have, for example
$R>R_0$,
\begin{align*}
 \frac{1}{R} &\sum_{k=mR}^{(m+1)R-1}  f\Big(\frac {k+\alpha} a -
\Big\lfloor\frac {k+\alpha} a \Big\rfloor\Big)\leq \frac{1}{R} \sum_{k=mR}^{(m+1)R-1}  h\Big(\frac {k+\alpha} a -
\Big\lfloor\frac {k+\alpha} a \Big\rfloor\Big)\leq \int h + \frac \epsilon 3 \\ &
=\int_{I^c} f + \int_I h -\int_I f+\int_I f + \frac \epsilon 3 \leq \int f +2M\frac \epsilon {3M} +\frac \epsilon 3=\int f+\epsilon.
\end{align*}

\end{proof}

\subsection{Proof of Theorem~\ref{thm:shiftedlattice1}}

\begin{proof}
Part (a) follows immediately from Theorem~\ref{thm:translateanddilate}.

For part (b), let $J\subseteq\{1,\,2,\,\dots,\,N\} = \{k_j\}_{j=1}^{|J|}$ and assume without loss of generality that $I=[0,|J|]$.
First we will show that $\displaystyle{\E(\cup_{j=1}^{|J|}\Gamma_{k_j})}$ satisfies (\ref{eqn:RB}).  To that end, 
let $\{c^j_n\colon j\in J,\,n\in\Z\}$ be a finite sequence and consider the sum
$$\sum_{j=1}^{|J|}\,\sum_{n\in\Z} c^j_n\,e^{2\pi i(Nn+k_j)t} = \sum_{j=1}^{|J|} e^{2\pi i k_j t}\,\sum_{n\in\Z} c^j_n\,e^{2\pi inN t}
  =  \sum_{j=1}^{|J|} e^{2\pi i k_j t}\,f_j(t),$$
  and note that each $f_j(t)$ has period $1/N$.
  
Now,
\begin{align*}
\bigg\|\sum_{j=1}^{|J|}e^{2\pi i k_j t}\,f_j(t)\bigg\|^2
=& \int_I\bigg|\sum_{j=1}^{|J|}e^{2\pi i k_j t}\,f_j(t)\bigg|^2\,dt \\
=& \sum_{m=0}^{|J|-1}\int_0^{1/N}\bigg|\sum_{j=1}^{|J|}e^{2\pi i k_j (t+m/N)}\,f_j(t+m/N)\bigg|^2\,dt \\
=& \sum_{m=0}^{|J|-1}\int_0^{1/N}\bigg|\sum_{j=1}^{|J|}e^{2\pi i k_j m/N}\,e^{2\pi i k_j t}\,f_j(t)\bigg|^2\,dt.
\end{align*}
The $|J|\times|J|$ Vandermonde matrix $V$ given by $V = (v_{j,m})= (e^{2\pi i m k_j/N})$, $j=1,\dots,|J|,\,m=0,\dots,|J|-1$
is invertible and therefore there exist constants, $A,\,B>0$ such that for all $\x\in\C^{|J|}$,
$A\|\x\|^2_2 \le \|V\x\|^2_2 \le B\,\|\x\|^2_2$, and $A\|\x\|^2_2 \le \|V^*\x\|^2_2 \le B\,\|\x\|^2_2$
that is, 
\begin{equation}\label{eqn:vandermonde}
A\,\sum_{j=1}^{|J|} |x_j|^2 \le \sum_{m=0}^{|J|-1}\bigg|\sum_{j=1}^{|J|} v_{j,m}\,x_j\bigg|^2 \le B\,\sum_{j=1}^{|J|} |x_j|^2,\ \mbox{and}\ 
A\,\sum_{m=0}^{|J|-1} |x_m|^2\le \sum_{j=1}^{|J|}\bigg|\sum_{m=0}^{|J|-1} v_{j,m}\,x_m\bigg|^2 \le B\,\sum_{m=0}^{|J|-1} |x_m|^2.
\end{equation}
Therefore,
$$A\,\int_0^{1/N} \sum_{j=1}^{|J|}|e^{2\pi i k_j t}\,f_j(t)|^2\,dt \le
\sum_{m=0}^{|J|-1}\int_0^{1/N} \bigg|\sum_{j=1}^{|J|} e^{2\pi i k_j m/N}\,e^{2\pi i k_j t}\,f_j(t)\bigg|^2\,dt 
\le B\,\int_0^{1/N} \sum_{j=1}^{|J|}|e^{2\pi i k_j t}\,f_j(t)|^2\,dt$$
so that
$$A\,\int_0^{1/N} \sum_{j=1}^{|J|}\bigg|\sum_{n\in\Z} c^j_n\,e^{2\pi i(Nn+k_j)t}\bigg|^2\,dt \le
\bigg\|\sum_{j=1}^{|J|}e^{2\pi i k_j t}\,f_j(t)\bigg\|^2
\le  B\,\int_0^{1/N} \sum_{j=1}^{|J|}\bigg|\sum_{n\in\Z} c^j_n\,e^{2\pi i(Nn+k_j)t}\bigg|^2\,dt,$$
which implies by the orthogonality of $\{e^{2\pi iNn t}\colon n\in\Z\}$ in $L^2[0,1/N]$ that
$$\frac{A}{N}\,\sum_{j=1}^{|J|}\sum_{n\in\Z} |c^j_n|^2 \le 
\bigg\|\sum_{j=1}^{|J|}\,\sum_{n\in\Z} c^j_n\,e^{2\pi i(Nn+k_j)t}\bigg\|^2 \le\frac{B}{N}\,\sum_{j=1}^{|J|}\sum_{n\in\Z} |c^j_n|^2$$
which is (\ref{eqn:RB}).

It remains to show that $\displaystyle{\E(\cup_{j\in J}\Gamma_{k_j})}$ is complete in $L^2(I)$ which can be accomplished
with a calculation similar to the above, viz.,
\begin{align*}
\sum_{j=1}^{|J|}\sum_{n\in\Z} |\langle f, e^{2\pi i (Nn+k_j)}\rangle|^2
=& \sum_{j=1}^{|J|}\sum_{n\in\Z}\bigg|\int_I f(t)\,e^{-2\pi iNn t}\,e^{-2\pi ik_j t}\,dt\bigg|^2 \\
=& \sum_{j=1}^{|J|}\sum_{n\in\Z}\bigg|\sum_{m=0}^{|J|-1}\int_0^{1/N} f(t+m/N)\,e^{-2\pi iNn (t+m/N)}\,e^{-2\pi ik_j (t+m/N)}\,dt\bigg|^2 \\
=& \sum_{j=1}^{|J|}\sum_{n\in\Z}\bigg|\int_0^{1/N}e^{-2\pi ik_j t} \bigg[\sum_{m=0}^{|J|-1}f(t+m/N)\,e^{-2\pi ik_j m/N}\bigg]\,e^{-2\pi iNn t}\,dt\bigg|^2 \\
=& \frac{1}{N}\sum_{j=1}^{|J|}\int_0^{1/N}\bigg|\sum_{m=0}^{|J|-1}f(t+m/N)\,e^{-2\pi ik_j m/N}\bigg|^2\,dt
\end{align*}
by Parseval's formula.  Applying (\ref{eqn:vandermonde}) gives
\begin{align*}
  \frac{A}{N}\int_I |f(t)|^2\,dt &= \frac{A}{N}\int_0^{1/N}\sum_{m=0}^{|J|-1}|f(t+m/N)|^2\,dt \\
  &\le \frac{1}{N}\sum_{j=1}^{|J|}\int_0^{1/N}\bigg|\sum_{m=0}^{|J|-1} f(t+m/N)\,e^{-2\pi ik_j m/N}\bigg|^2\,dt 
  = \sum_{j=1}^{|J|}\sum_{n\in\Z} |\langle f, e^{2\pi i (Nn+k_j)}\rangle|^2 \\ 
  &\le \frac{B}{N}\sum_{m=0}^{|J|-1}\int_0^{1/N}|f(t+m/N)|^2\,dt = \frac{B}{N}\int_I |f(t)|^2\,dt.
\end{align*}
This shows that $\displaystyle{\E(\cup_{j=1}^{|J|}\Gamma_{k_j})}$ is a frame in $L^2(I)$ and hence complete.
\end{proof}

\bibliographystyle{plain}

\end{document}